\documentclass{amsart}
\usepackage{amsmath}
\usepackage{amssymb}
\usepackage{graphicx}
\usepackage{latexsym}
\usepackage{color}
\usepackage{amscd}
\usepackage[all]{xy}
\usepackage{enumerate}
\usepackage{hyperref}
\usepackage{subfigure}
\usepackage{soul}
\newtheorem{thm}{Theorem}

\newtheorem{theorem}[thm]{Theorem}
\newtheorem{lemma}[thm]{Lemma}

\newtheorem{question}[thm]{Question}

\theoremstyle{definition}

\newtheorem*{definition*}{Definition}

\newtheorem{remark}[thm]{Remark}

\newcommand{\CPb}{\overline{\mathbb{CP}}{}^{2}}
\newcommand{\CP}{{\mathbb{CP}}{}^{2}}
\newcommand{\CPo}{{\mathbb{CP}}{}^{1}}

\newcommand{\N}{\mathbb{N}}
\newcommand{\Z}{\mathbb{Z}}

\newcommand{\K}{{\rm K3}}
\newcommand{\W}{\widetilde}

\def \x {\times}
\def \eu{{\text{e}}}

\begin{document}

\title[Small symplectic Calabi-Yaus and exotic $4$-manifolds via  pencils]
{Small symplectic Calabi-Yau surfaces \\ and exotic $4$-manifolds via genus-$3$  pencils }

\author[R. \.{I}. Baykur]{R. \.{I}nan\c{c} Baykur}
\address{Department of Mathematics and Statistics, University of Massachusetts, Amherst, MA 01003-9305, USA}
\email{baykur@math.umass.edu}

\begin{abstract}
We explicitly produce symplectic genus-$3$ Lefschetz pencils (with base points), whose total spaces are  homeomorphic but not diffeomorphic to rational surfaces $\CP \# p \, \CPb$ for $p= 7, 8, 9$. We then give a new construction of an infinite family of symplectic Calabi-Yau surfaces with first Betti number $b_1=2,3$, along with a surface with $b_1=4$ homeomorphic to the $4$-torus. These are presented as the total spaces of symplectic genus-$3$ Lefschetz pencils we construct via new positive factorizations in the mapping class group of a genus-$3$ surface. Our techniques in addition allow us to answer in the negative a question of Korkmaz  regarding the upper bound on $b_1$ of a genus-$g$ fibration.
\end{abstract}

\maketitle

\setcounter{secnumdepth}{2}
\setcounter{section}{0}

\vspace{-0.2in}

\section{Introduction}

Since the advent of Gauge theory, many new construction techniques, such as rational blowdowns, generalized fiber sum, knot surgery and Luttinger surgery, have been introduced and successfully employed to produce symplectic $4$-manifolds homeomorphic but not diffeomorphic to rational surfaces $\CP \# p \, \CPb$, for $p \geq 2$. (e.g.~\cite{ABP, AP, BK, Donaldson87, FSrationalblowdown, FSKnotsurgery, FSPinwheels, FSReverseEngineering, FriedmanMorgan, Gompf, Kotschick, ParkD, ParkJ, StipsiczSzabo}.) All these $4$-manifolds should admit symplectic \textit{Lefschetz pencils} (which, by definition, always have base points) by the ground-breaking work of Donaldson \cite{Donaldson}. Despite this a priori existence result, no explicit Lefschetz pencils on these \emph{small} (as in \emph{small} second homology) exotic $4$-manifolds were known up to date. 

The main goal of this article is to introduce novel ways of constructing \emph{small} positive factorizations (as in \emph{small} number of positive Dehn twists) of the boundary multi-twist in the mapping class group $\Gamma_g^m$ of a compact genus-$g$ surface with $m>0$ boundary components, which correspond to monodromy factorizations of genus-$g$ Lefschetz pencils on small symplectic $4$-manifolds. A first application yields the first explicit examples of pencils on exotic rational surfaces: 

\begin{theorem} \label{main1}
There are symplectic genus-$3$ Lefschetz pencils $(X_i, f_i$) whose total spaces are homeomorphic but not diffeomorphic to $\CP \# (6+i) \CPb$ for $i=1, 2, 3$. They realize the smallest possible genera pencils in these homeomorphism classes.
\end{theorem}

\noindent A key result leading to Theorem~\ref{main1} is that $\CP \# \, p \CPb$, for $p \leq 9$, does not admit a genus $g \geq 2$ Lefschetz fibration or a genus-$g$ Lefschetz pencil with $m < 2g-2$ base points (Lemma~\ref{notrational}), which provides a new criterion to argue that the total spaces of such pencils/fibrations like the ones in Theorem~\ref{main1} are exotic. We moreover observe that $g=3$ is the smallest possible genus for a Lefschetz pencil on these exotic $4$-manifolds (Remark~\ref{smallest}). We hope that the additional information on these exotic symplectic $4$-manifolds we construct, namely the existence of genus-$3$ pencils on them, will help with understanding whether there are distinct minimal symplectic $4$-manifolds in the  homeomorphism classes $\CP \# \, p \CPb$, for $p < 9$ (a question which is still open as of today; see Remark~\ref{distinctsymp}).

To produce the positive factorizations for the pencils in Theorem~\ref{main1}, we employ a new way of deriving a positive factorization $\W{W}$ in $\Gamma_g^m$, $m \geq 0$ from a collection of positive factorizations \,$W_1, \ldots, W_k$ of \emph{boundary twists} in $\Gamma_h^n$\,, with $h< g$ and $n>0$. For $(X,f)$ the Lefschetz pencil\,/\,fibration corresponding to $\W{W}$ and $(X_i, f_i)$ to $W_i$, we say $(X,f)$ is obtained by \emph{breeding} $(X_1,f_1), \ldots, (X_k, f_k)$. This involves embedding the latter into $\Gamma_g^m$ as factorizations for \emph{achiral} Lefschetz pencils\,/\,fibrations, and then canceling matching pairs of positive and negative Dehn twists. The idea for this construction scheme comes from the smallest hyperelliptic genus-$3$ Lefschetz \emph{fibration} Mustafa Korkmaz and the author produced in \cite{BaykurKorkmazGenus3}. 

A careful variation of our construction of $(X_i, f_i)$ yield to our next theorem:

\begin{theorem} \label{main2}
For any pair of non-negative integers $M=(m_1, m_2)$, there exists a \linebreak  symplectic genus-$3$ Lefschetz pencil $(X_{i,M}, f_{i,M})$ with $c_1^2(X_{i,M})=3-i$, \mbox{$\chi(X_{i,M})=1$}, and $\pi_1(X_{i,M})=  (\Z \, / {m_1 \, \Z}) \oplus (\Z \, / {m_2 \, \Z})$, for each $i=1,2,3$. Infinitely many of these pencils have total spaces homotopy inequivalent to a complex surface.
\end{theorem}

\noindent In contrast, I know of only one explicit example of a non-holomorphic genus-$3$ Lefschetz \emph{pencil} in the literature, due to Ivan Smith. (\cite{SmithGenus3}[Theorem~1.3]; see also Remark~\ref{nonholompencil}.) Our construction technique yielding the pencils in Theorem~\ref{main2} can be easily generalized to obtain many other non-holomorphic pencils. 

The second part of our article will deal with new constructions of symplectic Calabi-Yau surfaces. Recall that a symplectic $4$-manifold $(X, \omega)$ is called a \textit{symplectic Calabi-Yau surface} (SCY) if it has trivial canonical class $K_{X} \in H^2(X ; \Z)$, in obvious analogy with complex Calabi-Yau surfaces. As shown by Tian-Jun Li, any \textit{minimal} symplectic $4$-manifold with Kodaira dimension zero is a symplectic Calabi-Yau surface or has trivial canonical class \cite{LiSCY}. Up to date, the only known examples of symplectic Calabi-Yau surfaces that are not diffeomorphic to complex Calabi-Yau surfaces are torus bundles over tori (at least \emph{virtually})\footnote{A couple other examples of SCYs, such as symplectic $S^1$-bundles over $T^2$-bundles over $S^1$, although not known to be torus bundles themselves, are finitely covered by them \cite{FriedlVidussi}.} , leading to the question  \cite{LiSCY, Donaldson_SCY}:

\begin{question} \label{SCYclassification}
Is every symplectic Calabi-Yau surface with $b_1> 0$ diffeomorphic to an oriented torus bundle over a torus?
\end{question}

Works of Tian-Jun Li and Stefan Bauer independently established that any symplectic Calabi-Yau surface with $b_1 >0$  has the same rational homology as a torus bundle (and that of $\K$ and the Enriques surface when $b_1=0$)   \cite{LiSCY, LiQuaternionic, Bauer}. All torus bundles over tori have \emph{solvmanifold fundamental groups} \cite{Hillman}, and notably, it was shown by Stefan Friedl and Stefano Vidussi that the \textit{homeomorphism type} of a symplectic Calabi-Yau surface with a solvmanifold fundamental group is uniquely determined \cite{FriedlVidussi}. As stated by Tian-Jun Li \cite{LiSCYSurvey}, what seems to commonly provide a posteriori reasoning for a positive answer to the Question~\ref{SCYclassification} is the lack of any new constructions of symplectic $4$-manifolds of Kodaira dimension zero. Surgical operations widely used to construct most interesting symplectic $4$-manifolds in the past, such as Luttinger surgery, generalized fiber sums, knot surgery, or simplest rational blow-down operations, do not produce any new SCYs \cite{HoLi, LiSCYSurvey, UsherKodaira, Dorfmeister}.

We will present  a new construction of symplectic Calabi-Yau surfaces with \mbox{$b_1>0$} \emph{via Lefschetz pencils}, where the genus of the pencil $(X,f)$ versus $c_1^2(X)$, building on the work of Cliff Taubes, will determine the Kodaira dimension of $X$ (Lemma~\ref{SCYpencil}; also \cite{SatoKodaira}). Once again, we will focus on generating the smallest genus examples: genus-$3$ pencils. To produce our examples, we will use a version of the aforementioned breeding technique for producing new positive factorizations in $\Gamma_3^m$, which we will call \emph{inbreeding}, where this time copies of the same positive factorization in $\Gamma_2^n$ will be embedded into $\Gamma_3^m$ by identifying some of the boundary components of the fiber $\Sigma_2^n$ and introducing extra $\pi_1$-generators this way. The result is a family of positive factorizations
\[W_{\phi}=
t_{\phi(B_{0})} t_{\phi(B_{1})} t_{\phi(B_{2})} \, t_{\phi(A_{0})} t_{\phi(A_{1})} t_{\phi(A_{2})} \, t_{B'_{0}} t_{B'_{1}} t_{B'_{2}} \, t_{A'_{0}} t_{A'_{1}} t_{A'_{2}} = 1\]
in $\Gamma_3$, where $\phi \in \Gamma_3$ is parametrized by pairs of mapping classes $\phi_1, \phi_2 \in \Gamma_1^2$, and the curves $B_{j}, A_{j}$ are as shown in Figure~\ref{SCYcurves}. Here we inbreed the monodromy factorizations of Matsumoto's well-known genus-$2$ Lefschetz fibration \cite{Matsumoto}, and using further lifts of it obtained by Noriyuki Hamada \cite{Hamada}, we can even derive an explicit lift of $W_{\phi}$, which is a positive factorization of the boundary multi-twist $\Delta = t_{\partial_1} t_{\partial_2} t_{\partial_3} t_{\partial4}$ in $\Gamma_3^4$. Letting $(X_{\phi}, f_{\phi})$ denote the corresponding symplectic genus-$3$ Lefschetz pencil, we obtain (Theorem~\ref{SCYfamily}):

\begin{thm} \label{main3}
Positive factorizations $W_{\phi}$ \,prescribe symplectic genus-$3$ Lefschetz pencils on a family of symplectic Calabi-Yau surfaces $X_{\phi}$\, in all rational homology classes of torus bundles over tori.
\end{thm}

\noindent This provides a first step towards analyzing the monodromies of pencils on a rich family of SCYs, as proposed by Simon Donaldson in \cite{Donaldson_MCG}[Problem~5].

The first question that arises here is whether or not all $X_{\phi}$ are diffeomorphic to torus bundles over tori, in reference to Question~\ref{SCYclassification}. Torus bundles over tori \emph{with sections} do admit symplectic genus-$3$ pencils \cite{SmithTorus}; however the uncanny freedom we have in the choice of $\phi$ appears to exceed that of a torus bundle with a section (see the discussion in Remark~\ref{CompareSmith}). The explicit nature of our construction allows us to derive explicit presentations for $\pi_1(X_{\phi})$, which for instance yields an $X$ (when $\phi$ is trivial) that is --at least-- homeomorphic to the $4$-torus. We will discuss this example in complete detail in Section~\ref{Sec:4torus}. It is not clear to us at this point whether or not \emph{all} $\pi_1(X_{\phi})$ are $4$-dimensional solvmanifold groups (some of them certainly are), which we hope to understand in future work.

On the other hand, we observe that \emph{all} $X_{\phi}$ are obtained from $X$, our SCY homeomorphic to the $4$-torus, via \emph{fibered} Luttinger surgeries along Lagrangian tori or Klein bottles  (Remark~\ref{CompareHoLi}). So if our family of $X_{\phi}$ is equal to the family of  torus bundles over tori, it immediately confirms an improved version\footnote{Only Luttinger surgeries along tori are considered in \cite{HoLi}, but it is natural to include surgeries along Klein bottles, too.} of an interesting conjecture by Tian-Jun Li and Chung-I Ho \cite{HoLi}[Conjecture~4.9]: \textit{Any smooth oriented torus bundle $X$ possesses a symplectic structure $\omega$ such that $(X, \omega)$ can be obtained by applying Luttinger surgeries to $(T^4, \omega_{\rm{std}})$} (Remark~\ref{CompareHoLi}).

Finally, Theorem~\ref{main3} implies that we have symplectic genus-$3$ pencils $(X_\phi, f_{\phi})$ with $H_1(X_{\phi}) \cong \Z^2 \oplus (\Z \, / {m_1 \, \Z}) \oplus (\Z \, / {m_2 \, \Z})$ for any $m_1, m_2 \in \N$. This yields further (infinitely many) examples of non-holomorphic genus-$3$ pencils, in the same fashion as in our earlier examples in Theorem~\ref{main2}. 

The techniques we employed to produce the genus-$3$ pencils in Theorems~\ref{main1},  and~\ref{main3} can be easily adapted to produce higher genera pencils on similar $4$-manifolds; we can in fact derive many of these by repeatedly breeding/inbreeding the same collection of positive factorizations used in this paper. This direction will be explored elsewhere. We will however discuss one immediate generalization of our construction of a genus-$3$ pencil to higher genera. It will allow us to address a problem raised by Mustafa Korkmaz \cite{KorkmazProblems}[Problem~2.7], which we re-express here for pencils (which are even more constrained than fibrations):

\begin{question} 
Is $b_1(X) \leq g$ for any non-trivial genus-$g$ Lefschetz pencil $(X,f)$? \linebreak  Is there an upper bound on $b_1(X)$ linear in $g$, and sharper than \,$2g-1$?
\end{question}

\noindent In Section~\ref{Sec:bound}, generalizing our construction of a genus-$3$ pencil on $X$ homeomorphic to the $4$-torus,  we will produce genus-$g$ Lefschetz pencils violating the above bound for every odd $g> 1$. Though not much; we will indeed revamp Korkmaz's question by raising the upper bound to $g+1$ (Question~\ref{b1RefinedQuestion}).

\enlargethispage{0.1in}
\vspace{0.2in}
\noindent \textit{Acknowledgements.} We would like to thank Weimin Chen, Bob Gompf, Jonathan Hillman, Tian-Jun Li, and Stefano Vidussi for helpful conversations related to our constructions in this paper. Thanks to Noriyuki Hamada for telling us about his preprint on the sections of the Matsumoto fibration. Our results on small symplectic Calabi-Yau surfaces via Lefschetz fibrations with exceptional sections were first presented at the Great Lakes Geometry Conference in Ann Arbor in March 2015; we would like to thank the organizers for motivating discussions. Last but not least, we are grateful to Mustafa Korkmaz for his interest in our work and numerous stimulating discussions. The author was partially supported by the NSF Grant DMS-$1510395$.

%\vspace{0.1in}
%======================================================
\section{Preliminaries}  \label{preliminaries}
%======================================================

Here we quickly review the definitions and basic properties of Lefschetz pencils and fibrations,  Dehn twist factorizations in mapping class groups of surfaces, and symplectic Calabi-Yau surfaces. The reader can turn to \cite{GompfStipsicz, LiKodairaSurvey, BaykurKorkmaz}  for more details.

\subsection{Lefschetz pencils and fibrations}  \
%======================================================

A \emph{Lefschetz pencil} on a closed, smooth, oriented $4$-manifold $X$ is a smooth surjective map  $f: X \setminus \{ b_j \} \to S^2$, defined on the complement of a non-empty finite collection of points $\{b_j \}$,  such that around every \emph{base point} $b_j$ and \emph{critical point} $p_i$ there are local complex coordinates (compatible with the orientations on $X$ and $S^2$) with respect to which the map $f$ takes the form  $(z_1,z_2) \mapsto z_1/z_2$ and $(z_1, z_2) \mapsto z_1 z_2$, respectively. A \emph{Lefschetz fibration} is defined similarly for $\{ b_j \} = \emptyset$. Blowing-up all the base points $b_j$ in a pencil $(X,f)$, one obtains a Lefschetz fibration $(\W{X},\W{f})$ with disjoint $(-1)$-sphere sections $S_j$ corresponding to $b_j$, and vice versa.

We say $(X,f)$ is a \emph{genus $g$ Lefschetz pencil} or \emph{fibration} for $g$ the genus of a \emph{regular fiber} $F$ of $f$. The fiber containing the critical point $p_i$ has a nodal singularity at $p_i$, which locally arises from shrinking a simple loop $c_i$ on $F$, called the \emph{vanishing cycle}. A singular fiber of  $(X,f)$ is called \emph{reducible}  if $c_i$ is separating. When $c_i$ is null-homotopic on $F$, one of the fiber components becomes an \emph{exceptional sphere}, an embedded $2$-sphere of self-intersection $-1$, which we can  blow down without changing the rest of the fibration. 

In this paper we use the term Lefschetz fibration only when the set of critical points $\{p_i\}$ is non-empty, i.e. when the Lefschetz fibration is \emph{non-trivial}. We moreover assume that the fibration is \emph{relatively minimal}, i.e. it there are no exceptional spheres contained in the fibers, and also that the points $p_i$ lie in distinct \emph{singular fibers}, which can be always achieved after a small perturbation.

A widely used way of constructing Lefschetz fibrations is the \emph{fiber sum} operation: if $(X_i, f_i)$ are genus-$g$ Lefschetz fibrations with regular fiber $F_i$ for $i=1,2$, then their \emph{fiber sum} is a genus-$g$ Lefschetz fibration $(X,f)$, where $X$ is obtained by removing a fibered tubular neighborhood  of each $F_i$ from $X_i$ and then identifying the resulting boundaries via a fiber-preserving, orientation-reversing self-diffeomorphism, and where $f$ restricts to $f_i$ on each $X_i \setminus \nu  F_i$. A Lefschetz fibration $(X,f)$ that can be expressed in this way is called \emph{decomposable}, and each $(X_i, f_i)$ are called its \emph{summands}.

Allowing local models $(z_1, z_2) \mapsto z_1 \bar{z}_2$ around the critical points $p_i$, called \emph{negative} critical points, we can extend all of the above to \emph{achiral} Lefschetz pencils and fibrations.

\subsection{Positive factorizations} \
%======================================================

Let $\Sigma_g^m$ denote a compact, connected, oriented surface genus $g$ with $m$ boundary components. We denote by $\Gamma_g^m$ its \emph{mapping class group}; the group composed of orientation-preserving self homeomorphisms of $\Sigma_g^m$ which restrict to the identity along $\partial \Sigma_g^m$, modulo isotopies that also restrict to the identity along $\partial \Sigma_g^m$. We write $\Sigma_g = \Sigma_g^0$, and $\Gamma_g=\Gamma_g^0$ for simplicity. Denote by $t_c \in \Sigma_g^m$ the positive (right-handed) Dehn twist along the simple closed curve $c \subset \Sigma_g^m$. The inverse $t^{-1}_c$ denotes the negative (left-handed) Dehn twist along $c$. 

Let $\{c_i\}$ be a \emph{non-empty} collection of simple closed curves on $\Sigma_g^m$, which do not become null-homotopic when $\partial \Sigma_g^m$ is capped off by disks, and let $\{\delta_j\}$ be a collection of curves parallel to distinct boundary components of $\Sigma_g^m$. If the relation
\begin{equation} \label{factorization}
t_{c_l} \cdots t_{c_2} t_{c_1} = t_{\delta_1} \cdots t_{\delta_m} \, \, 
\end{equation}
holds in $\Gamma_g^m$, we call the word $W$ on the left-hand side a \emph{positive factorization} of the boundary multi-twist $\Delta = t_{\delta_1} \cdots t_{\delta_m}$ in $\Gamma_g^n$. (We will also use $\partial_i$ instead of $\delta_i$ to denote boundary components at times when there are several surfaces with boundaries involved in our discussion.) Capping off $\partial \Sigma_g^m$ induces a homomorphism  $\Gamma_g^m \to \Gamma_g$, under which $W$ maps to a similar positive factorization of the identity element $1 \in \Gamma_g$. 

The positive factorization in (\ref{factorization}) gives rise to a genus-$g$ Lefschetz fibration $(\W{X},\W{f})$ with $m$ disjoint $(-1)$-sections $S_j$, and equivalently a genus-$g$ Lefschetz pencil $(X,f)$ with $m$ base points. Identifying the regular fiber $F$ with $\Sigma_g$, we can view the vanishing cycles as $c_i$. In fact, every Lefschetz fibration or pencil can be described by a positive factorization $W$ as in (\ref{factorization}) \cite{GompfStipsicz, Kas, Matsumoto}, which is called the \emph{monodromy factorization} of $\W{f}$ or $f$.

Let $W$ be a positive factorization of the form\, $W= P P'$ in $\Gamma_g^m$, where $P, P'$ are products of positive Dehn twists along curves which do not become null-homotopic when $\partial \Sigma_g^m$ is capped off by disks. If $P= \prod t_{c_i}$, as a mapping class, commutes with some $\phi \in \Gamma_g^m$, we can produce a new positive factorization $W_{\phi}= P^{\phi} P'$, where $P^{\phi}$ denotes the conjugate factorization $\phi P \phi^{-1} =  \prod (\phi \, t_{c_i} \, \phi^{-1}) = \prod  t_{\phi(c_i)} \, . $ In this case, we say $W_{\phi}$ is obtained from $W$ by a \emph{partial conjugation} $\phi$ along the factor $P$. When $P, P \in \Gamma_g$ are positive factorizations, $W=P P'$ is another positive factorization, Lefschetz fibration corresponding to which is a fiber sum of the fibrations corresponding to $W'$ and $W''$. 
%(Whereas a product of positive factorizations of boundary multi-twists never result in one.) 

Allowing \emph{negative} Dehn twists, which correspond to negative critical points, we can more generally work with \emph{factorizations} for achiral Lefschetz fibrations and pencils. All of the above extend to this general case.

\subsection{Symplectic $4$-manifolds and the Kodaira dimension}  \
% =================================================

It was shown by Donaldson that every symplectic $4$-manifold $(X, \omega)$ admits a \emph{symplectic} Lefschetz pencil whose fibers are symplectic with respect to $\omega$ \cite{Donaldson}. Conversely, generalizing a construction of Thurston, Gompf showed that the total space of a (nontrivial) Lefschetz fibration, and in particular blow-up of any pencil, always admits a symplectic form $\omega$ with respect to which all regular fibers and any preselected collection of disjoint sections are symplectic \cite{GompfStipsicz}. Whenever we take a symplectic form $\omega$ on a Lefschetz pencil or fibration $(X,f)$ we will assume that it is of Thurston-Gompf type, with respect to which any explicitly discussed sections of $f$ will always be assumed to be symplectic as well.

The Kodaira dimension for projective surfaces can be extended to symplectic $4$-manifolds as follows: Let $K_{X_{\text{min}}}$ be the canonical class of  a minimal model $(X_{\text{min}}, \omega_{\text{min}})$ of $(X, \omega)$. The \textit{symplectic Kodaira dimension} of $(X, \omega)$, denoted by $\kappa=\kappa(X,\omega)$ is defined as 
\[
\kappa(X,\omega)=\left\{\begin{array}{rl}-\infty& \mbox{\ \ if
}K_{X_{\text{min}}}\cdot[\omega_{\text{min}}]<0 \mbox{ or } K_{X_{\text{min}}}^{2}<0 \\
0 & \mbox{\ \ if } K_{X_{\text{min}}}\cdot[\omega_{\text{min}}]= K_{X_{\text{min}}}^{2}=0\\ 1 &
\mbox{\ \ if } K_{X_{\text{min}}}\cdot[\omega_{\text{min}}]>0\mbox{ and
} K_{X_{\text{min}}}^{2}=0\\2& \mbox{\ \ if } K_{X_{\text{min}}}\cdot[\omega_{\text{min}}]>0\mbox{
and } K_{X_{\text{min}}}^{2}>0\end{array}\right .
\]
Remarkably, $\kappa$ is not only independent of the minimal model $(X_{\text{min}}, \omega_{\text{min}})$ but also of the chosen symplectic form $\omega$ on $X$: it is a smooth invariant of the $4$-manifold $X$ \cite{LiSCY}. Symplectic $4$-manifolds with $\kappa = -\infty$ are classified up to symplectomorphisms, which are precisely the rational and ruled surfaces \cite{LiKodairaSurvey}. Symplectic $4$-manifolds with $\kappa=0$ are the analogues of the Calabi-Yau surfaces that have torsion canonical class \cite{LiSCY}. It has been shown by Tian-Jun Li, and independently by Stefan Bauer \cite{LiSCY, Bauer}, that the rational homology type of any \textit{minimal} symplectic Calabi-Yau surface is that of either a torus bundle over a torus, the $\K$ surface, or the Enriques surface.

%\newpage
\vspace{0.1in}
% =========================================================
\section{Exotic symplectic rational surfaces via genus-$3$ pencils}
% =========================================================

Here we construct explicit positive factorizations for symplectic genus-$3$ Lefschetz pencils, whose total spaces are homeomorphic but not diffeomorphic to the complex rational surfaces.

\subsection{Breeding genus-$3$ Lefschetz pencils from genus-$2$ pencils} \
% =========================================================

In \cite{BaykurKorkmaz}, we have obtained the following relation in $\Gamma_2^1$
\begin{eqnarray*}
t_ e     t_{x_1}  t_{x_2} t_{x_3}   t_d  t_C  t_{ x_4 }  &=&  t_\delta \, ,
\end{eqnarray*}
which can be rewritten as
\begin{eqnarray*}
t_ e     t_{x_1}  t_{x_2} t_{x_3}   t_d t_{B_2}   t_C  t_\delta^{-1}  &=&   1 \, ,
\end{eqnarray*}
where $B_2=  t_c(x_4)$. Embed this relation into $\Gamma_3^1$ so that $\delta=\partial \Gamma_2^1$ is mapped to $c'$ and the remaining curves are as in Figure~\ref{Exoticcurves}, denoted by the same letters. (So, along with the embedding of $x_4$, the original curves on $\Gamma_2^1$ can be all seen in the obvious subsurface of $\Gamma_3^1$ in Figure~\ref{Exoticcurves}.) We thus get the following relation in $\Gamma_3^1$:
\begin{eqnarray} \label{embed1}
t_ e     t_{x_1}  t_{x_2} t_{x_3}   t_d t_{B_2}   t_C  t_{C'}^{-1} &=&   1 \, ,
\end{eqnarray}
or $P_1 \,  t_C  t_{C'}^{-1} =   1$, where we set $\underline{P_1= t_ e     t_{x_1}  t_{x_2} t_{x_3}   t_d t_{B_2}}$. Note that $P_1, t_C$ and $t_{C'}$ all commute with each other.

\begin{figure}[p!]
 \centering
     \includegraphics[width=12cm]{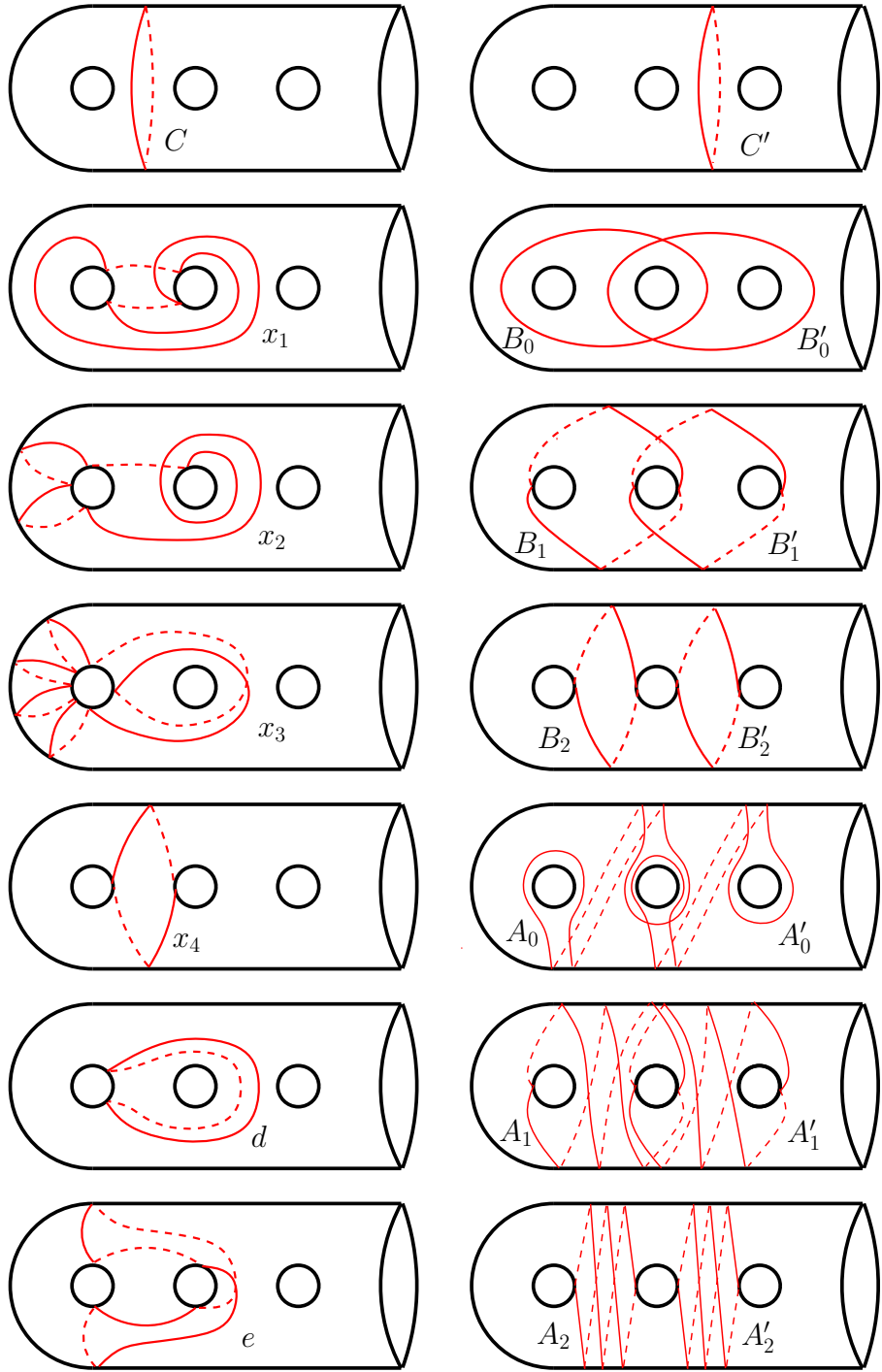}
\vspace{0.1in}
     \caption{Curves $C, x_1, x_2, x_3, x_4, d, e$ of the first embedding are given on the left. Curves $B_0, B_1, B_2, A_0, A_1, A_2$ of the second embedding (along with the same $C$) and $C', B'_0, B'_1, B'_2, A'_0, A'_1, A'_2$ of the third embedding are on the right. }
     \label{Exoticcurves}
\end{figure}

Our second building block will be a positive factorization of the boundary multi-twist in $\Gamma_2^2$, which is a \emph{lift} of the monodromy factorization of Matsumoto's well-known genus-$2$ Lefschetz fibration \cite{Matsumoto}. Such a lift to $\Gamma_2^1$ was first found by Ozbagci and Stipsicz in \cite{OzbagciStipsicz2} and a further lift to $\Gamma_2^2$ by Korkmaz in \cite{Korkmaz2}. We will however work with another lift discovered by Noriyuki Hamada, which yields the following relation in $\Gamma_2^2$ \cite{Hamada}:
\begin{eqnarray} \label{firstlift}
(t_{B_0} t_{B_1} t_{B_2} t_{C} )^2 &=& t_{\delta_1} t_{\delta_2}  \, ,
\end{eqnarray}
where $\delta_i$ are the boundary parallel curves, and the curves $B_i$ and $C$ are as shown on the left-hand side of Figure~\ref{Matsumoto}. We can rewrite this relation in $\Gamma_2^2$ as
\begin{eqnarray*}
t_{B_0} t_{B_1} t_{B_2} t_{A_0} t_{A_1} t_{A_2} t_{C}^2 t_{\delta_1}^{-1} t_{\delta_2}^{-1} &=&  1  \, ,
\end{eqnarray*}
where each $A_j= t_C(B_j)$ for $j=0,1,2$ (Figure~\ref{Matsumoto}). 

Let us first view this relation in $\Gamma_2^1$ by capping off one of the boundary components, say $\delta_1$. Embed the resulting relation into $\Gamma_3^1$, as induced by the same embedding of $\Gamma_2^1$ into $\Gamma_3^1$ we used above so that the boundary is mapped to $c'$, and all the other curves are as shown in Figure~\ref{Exoticcurves}, once again denoted by the same letters for simplicity. So the following holds in $\Gamma_3^1$:
\begin{eqnarray} \label{embed2}
t_{B_0} t_{B_1} t_{B_2} t_{A_0} t_{A_1} t_{A_2} t_{C}^2 t_{C'}^{-1} &=&  1  \, ,
\end{eqnarray}
or $P_2 \, t_{C}^2 t_{C'}^{-1}=   1$, where $\underline{P_2= t_{B_0} t_{B_1} t_{B_2} t_{A_0} t_{A_1} t_{A_2}} $. Here $P_2, t_C$ and $t_{C'}$ all commute with each other.

Finally, we will embed the lat relation we had above in $\Gamma_2^2$ into $\Gamma_3^1$ using an embedding of $\Sigma_2^2$ into $\Sigma_3^1$ such that $\delta_1$ is mapped to $c$, $\delta_2$ is mapped to $\delta=\partial \Sigma_3^1$, and the remaining curves are as shown in Figure~\ref{Exoticcurves}, this time their labels are decorated by prime. This gives  a third relation in $\Gamma_3^1$:
\begin{eqnarray} \label{embed3} 
t_{B'_0} t_{B'_1} t_{B'_2} t_{A'_0} t_{A'_1} t_{A'_2} t_{C'}^2 t_{C}^{-1} &=&  t_{\delta}  \, ,
\end{eqnarray} 
or $P'_2 \, t_{C'}^2 t_{C}^{-1}=  t_{\delta} $, where $\underline{P'_2=t_{B'_0} t_{B'_1} t_{B'_2} t_{A'_0} t_{A'_1} t_{A'_2}} $. Similar to the above, $P'_2, t_C$ and $t_{C'}$ all commute with each other.

\smallskip
Now for $\phi= t_{b_1}^{-1} t_{a_2}$ (where $b_1, a_2$ are as in Figure~\ref{pi1generators}), and in fact for any mapping class in $\Gamma_3^1$ that keeps $C$ and $C'$ fixed, we have
\begin{equation*}
(P_1)^{\phi} P_1 P'_2 \, t_C  = (P_1)^{\phi} t_C  t_{C'}^{-1} \, P_1 t_C  t_{C'}^{-1}  \,  P'_2 \, t_{C'}^2 t_{C}^{-1}  = (P_1 t_C  t_{C'}^{-1})^{\phi} \, t_{\delta}  =t_\delta \, ,
\end{equation*}
where the first equality follows from the commutativity relations noted above, and the second equality follows from the relations~$(\ref{embed1})$--$(\ref{embed3})$. So $\underline{W_1=(P_1)^{\phi} P_1 P'_2  t_C}$ is a positive factorization of $t_{\delta}$ in $\Gamma_3^1$. Similarly, we get two more positive factorizations $\underline{W_2=(P_1)^{\phi} P_2 P'_2 \, t^2_C}$ and $\underline{W_3=(P_2)^{\phi} P_2 P'_2 \, t^3_C}$ of $t_{\delta}$ in $\Gamma_3^1$. 

Each $W_i$ prescribes a symplectic genus-$3$ Lefschetz fibration  $(\widetilde{X_i}, \widetilde{f_i})$ with a \linebreak $(-1)$-section, which can be blown-down to arrive at a symplectic genus-$3$ Lefschetz pencil $(X_i, f_i)$ with one base point. They are all obtained by \emph{breeding} the genus-$2$ Lefschetz pencils whose monodromies we have lifted to $\Gamma_3^1$.

\smallskip
\vspace{0.1in}
\subsection{Homeomorphism type of the pencils $(X_i, f_i)$} \ \label{homeo}
% =========================================================

We will prove the following:

\begin{lemma} \label{lemhomeo}
$X_i$ is homeomorphic to $\CP \# \, (6+i) \, \CPb$, for $i=1,2,3$. 
\end{lemma} 

\begin{proof}
The proof will follow from Freedman's celebrated work \cite{Freedman}, for which we will determine the Euler characteristic, the signature, and the spin type of $X_i$, while showing that $X_i$ is simply-connected for all $i=1,2,3$.

The Euler characteristic of $X_i$ is given by 
\[ \eu(X_i)= \eu(\widetilde{X_i})-1= (4-4\cdot 3+\ell(W_i))-1=\ell(W_i)-9 \, , \]
 where $\ell(W_i)$ is the number of Dehn twists in the positive factorization $W_i$. Since $\ell(P_1)=\ell(P_2)=\ell(P'_1)=6$,  we have $\ell(W_i)=18+i$, and in turn, $\eu(X_i)= 9+i$.

The signature of $X_i$ can be calculated as $\sigma(X_i)=-5-i$  using  Ozbagci's algorithm \cite{Ozbagci}. Here is a quick argument for this calculation: we have obtained the monodromy factorization of $(\widetilde{X_1}, \widetilde{f_1})$ after canceling a pair of $t_C$ and ${t_C}^{-1}$, and two pairs of $t_{C'}$ and ${t^{-1}_{C'}}$ in the factorization 
\[ (P_1 t_C  t_{C'}^{-1})^{\phi} \, P_1 t_C  t_{C'}^{-1}  \,  P'_2 \, t_{C'}^2 t_{C}^{-1} = 1 \,  \]
in $\Gamma_3$ (where we capped off the boundary). The latter is a factorization of a genus-$3$ \emph{achiral} Lefschetz fibration, which is a fiber sum of three achiral Lefschetz fibrations prescribed by $P_1 t_C  t_{C'}^{-1}$.  $P_1 t_C  t_{C'}^{-1}$, and $P'_2 \, t_{C'}^2 t_{C}^{-1}$.  Each one of these three fibrations is obtained by a \emph{section sum} of an achiral genus-$2$ Lefschetz fibration and a trivial torus fibration along a section of self-intersection zero, so by Novikov additivity, its signature is equal to that of the achiral genus-$2$ fibration. The hyperelliptic signature formula for genus-$2$ fibrations \cite{Matsumoto, Endo} allows us to easily calculate these as $-2$, $-2$, and $-3$, which add up to give signature $-7$ for the original genus-$3$ achiral Lefschetz fibration. The signature of $(\widetilde{X_1}, \widetilde{f_1})$ is the same as the signature of this achiral fibration, since canceling a pair of separating Dehn twists with opposite signs amounts to taking out a pair of $\pm 1$ contribution to the signature. Hence, we get $\sigma(X_1) = \widetilde{X_1}+1=-6$. By identical arguments, we get $\sigma(X_2)=-7$ and  $\sigma(X_3)=-8$. 

Clearly each $X_i$ has an odd intersection form: this can be easily seen from the presence of  the reducible fiber component split off by the separating vanishing cycle $C$ and not hit by the $(-1)$-sphere. (Alternatively, we can use Rokhlin's theorem after showing that $\pi_1(X_i)=1$.)

Lastly, we will show that $\pi_1(X_i)=\pi_1(\widetilde{X_i})=1$, by calculating the latter as the quotient of $\pi_1(\Sigma_3)$ by the subgroup normally generated by the vanishing cycles of $(\widetilde{X_i}, \widetilde{f_i})$, for each $i=1,2,3$. It suffices to find enough relations that kill the generators $a_1, b_1, a_2, b_2, a_3, b_3$ of $\pi_1(\Sigma_3)$ (which are as shown in Figure~\ref{pi1generators}), so we do not need to work with the full presentation.
  
Each positive factorization $W_i$ contains the factor $P'_2 \, t_{C}$. So the following relations  hold for the finite presentations we have for each $\pi_1(\widetilde{X_i})$:
\begin{eqnarray}
& \phantom{o} &   [a_1,b_1][a_2,b_2][a_3,b_3]=1,\label{eqn:0} \\
& \phantom{o} &  [a_1,b_1]=1,  \label{eqn:C}  \\
& \phantom{o} &   a_2a_3=1, \label{eqn:B'0} \\
& \phantom{o} &   a_2 \bar{b}_2 a_3 \bar{b}_3  =1, \label{eqn:B'1}  \\
& \phantom{o} &  b_3 b_2  =1, \label{eqn:B'2}	
\end{eqnarray}
where the relators ~$(\ref{eqn:C})$--$(\ref{eqn:B'2})$ come from the 
vanishing cycles $C, B'_0, B'_1, B'_2$, respectively. We have $a_2= \bar{a}_3$ from $(\ref{eqn:B'0})$ and $b_2= \bar{b}_3$ from $(\ref{eqn:B'2})$. Together with $(\ref{eqn:B'1})$, these imply $[a_2, b_2]=[a_3,b_3]=1$. Together with  $(\ref{eqn:C})$, we conclude $[a_j, b_j]=1$ for every $j=1,2,3$ (and  $(\ref{eqn:0})$ becomes a trivial relation). 

From the factor $P_1$, we get the following relations (among many others)
\begin{eqnarray}
& \phantom{o} &  a_1 (\bar{b}_1a_2b_2)^2=1,      \label{eqn:x1} \\  
& \phantom{o} &  a_1  \bar{b}^3_1 a_2 b_2  a_2=1,      \label{eqn:x2} \\  
& \phantom{o} &  a_1  \bar{b}^5_1 a_2 [b_2,  a_2]  b_1 a_2 =1, \label{eqn:x3} \\  
& \phantom{o} &  b_2 b_1 [b_3, a_3]   =1,  \label{eqn:B2}
\end{eqnarray}
induced by  the vanishing cycles $x_1, x_2, x_3$ and $B_2$, respectively. Adding these to the previous relators from the factor $P'_2$, we immediately see that $[a_3,b_3]=1$ and $(\ref{eqn:B2})$ imply $b_1=\bar{b}_2$. So $(\ref{eqn:x1})$ implies that $a_1$ can be generated by $a_2$ and $b_2=\bar{b}_1$. 

From the factor $P_2$ we get the following relations (again, among many others)
\begin{eqnarray}
& \phantom{o} &   a_1 a_2=1, \label{eqn:B0} \\
& \phantom{o} &    b_2 \bar{a}_2 b_1 \bar{a}_1 [b_3, a_3]  =1, \label{eqn:B1}  \\
& \phantom{o} & b_2 b_1  [b_3, a_3]  =1,  \nonumber
\end{eqnarray}
induced by the vanishing cycles $B_0, B_1$ and $B_2$, respectively. We get $a_1=\bar{a}_2$, and together with the relators from $P'_2$ we once again get $b_1=\bar{b}_2$, since $[a_3, b_3]=1$. 

Since the positive factorization $W_1$ contains the factor $P_1 P'_2  \, t_{C}$ and $W_2, W_3$ both contain the factor $P_2 P'_2 \, t_{C}$, the above discussion shows that every $\pi_1(\widetilde{X_i})$ is a factor of the abelian group $\Z^2$ generated by $a_2$ and $b_2$. (In fact no additional relators come from the remaining twists in $P_1, P_2$ or $P'_2$.) For the remaining relators coming from the conjugated factors $P_1^{\phi}$ or $P_2^{\phi}$, it therefore suffices to look at the relators they induce in $H_1(\Sigma_3)$. Moreover, this allows us to simply look at the homology classes of the vanishing cycles in these conjugated factors. (We haven't bothered to give explicitly for this very reason.) 

Without the conjugated factor, in each case we have the abelianized relations
\begin{equation} \label{ab1}
a_3=-a_2 \, \text{ and }  \, b_3=-b_2=b_1, 
\end{equation}
and depending on whether $W_i$ contains the factor $P_1$ or $P_2$, either 
\begin{align} 
& \phantom{o} & a_1+ 2a_2 +4b_2 &= 0  \ \ \  \text{for $W_1$, or}   \label{ab2-1} \\
& \phantom{o} & a_1+a_2 &= 0 \ \ \ \text{for $W_2$ and $W_3$}, \label{ab2-2}
\end{align}
where we used $(\ref{ab1})$ to simplify.

For our conjugation $\phi= t_{b_1}^{-1} t_{a_2}$, we easily check using the Picard-Lefschetz formula that we get the additional relator
\begin{equation} \label{ab3}
 b_1 + a_2 +b_2 =0, 
\end{equation}
from both $P_1^{\phi}$ and $P_2^{\phi}$ whereas, depending on whether it is $P_1^{\phi}$or $P_2^{\phi}$ we get either
\begin{align} 
& \phantom{o} & a_1 + b_1 + 6 a_2 + 4 b_2 &=0 \ \ \  \text{for $W_1$ and $W_2$, or}   \label{ab4-1} \\
& \phantom{o} & a_1+ b_1+a_2 &= 0 \ \ \ \text{for $W_3$}. \label{ab4-2}
\end{align}
Therefore, $(\ref{ab1})$ and $(\ref{ab3})$, which hold for all $\pi_1(\widetilde{X_i})$, imply $a_2=0$. The relators in $(\ref{ab2-1})$, $(\ref{ab2-2})$ and $(\ref{ab4-1})$, $(\ref{ab4-2})$  involved in a given $W_i$ then easily give $b_2=-b_1=0$, for each $i=1,2,3$. Hence in all cases we have got $\pi_1(X_i)=\pi_1(\widetilde{X_i})=1$.

By Freedman, we have $X_i$ homeomorphic to $\CP \# \, (6+i) \, \CPb$, for $i=1,2,3$. 
\end{proof}

\smallskip
\vspace{0.1in}
\subsection{Diffeomorphism type of the pencils $(X_i, f_i)$} \
% =========================================================

A quick amendment to the title: we will determine the ``non-diffeomorphism type'' of $(X_i, f_i)$, namely, we will prove that $X_i$ are not diffeomorphic to rational surfaces. This will follow from the next lemma.

\begin{lemma} \label{notrational}
The rational surface $\CP \# p \CPb$,  for $p \leq 9$, does not admit a genus $g \geq 2$ Lefschetz fibration or a genus-$g$ Lefschetz pencil with $m < 2g-2$ base points.
\end{lemma}

\begin{proof}
It suffices to show this for $X= \CP \# 9 \CPb$, since we can always blow-up on the fibers to produce a genus-$g$ Lefschetz fibration (resp. pencil) on $X$ from a given genus-$g$ Lefschetz fibration (resp. pencil) on $\CP \# p \CPb$, for any $p < 9$. Assume that $X$ is equipped with a  genus $g \geq 2$ Lefschetz fibration or a pencil with $m < 2g-2$ base points. For our arguments to follow, it will be convenient to allow $m$ to be non-negative so that $m=0$ marks the fibration case.

Let $F= a H - \sum_{i=1}^9 c_i E_i$ be the fiber class, where $H_2(X)$ is generated by $H$ and $E_1, \ldots, E_9$, with $H^2=1$, $E_i \cdot E_j=- \delta_{ij}$, and $H \cdot E_i=0$. Since $F^2=m$, we have 
\[ a^2= m+ \sum_{i=1}^9 c_i^2 \, . \]

We can equip $X$ with a Thurston-Gompf symplectic form $\omega$ which makes the fibers symplectic. Moreover, we can choose an $\omega$-compatible almost complex structure $J$, even a generic one in the sense of Taubes (see e.g \cite{UsherDS}), with respect to which $f$ is $J$-holomorphic (for a suitable choice of almost complex structure on the base $2$-sphere).
It was shown by Li and Liu \cite{LiLiu} that for a generic $\omega$-compatible $J$, the class $H$ in the rational surface $X$ has a $J$-holomorphic representative . Hence, $F$ and $H$ both have $J$-holomorphic representatives, which implies that $F \cdot H = a \geq 0$.

There is a unique symplectic structure on $X$ up to deformation and symplectomorphisms \cite{LiLiu}, so we can apply the adjunction formula to get
\[ 2g-2= F^2 + K \cdot F = m +(-3H + \sum_{i=1}^9 E_i) \cdot (aH - \sum_{i=1}^9 c_i E_i) = m -3a + \sum_{i=1}^9 c_i \, .\]

Since $a, m \geq 0$, and $g \geq 2$, from the above equalities we have 
\[ 3 a  = \sqrt{9 a^2}=\sqrt{9 (m+\sum_{i=1}^9 c_i^2) }  \geq \sqrt{9 (\sum_{i=1}^9 c_i^2) }  = \sqrt{(\sum_{i=1}^9 1) (\sum_{i=1}^9 c_i^2)}  \geq \sqrt{|\sum_{i=1}^9 c_i|^2} \, , \]
where the last inequality is by Cauchy-Schwartz, and in turn 
\[  3a \geq  \sqrt{|\sum_{i=1}^9 c_i|^2}  = |\sum_{i=1}^9 c_i| = | 2g-2-m +3a | = 2g-2 -m  +3a\, , \]
implying $m \geq 2g-2$. The contradiction shows that there exists no such symplectic surface $F$. In turn, there is no such fibration or pencil. 
\end{proof}

\smallskip
\noindent
In fact, if $X$ is a symplectic $4$-manifold in the homeomorphism class of a rational surface  $\CP \# p \CPb$ with $p \leq 9$, $X$ is exotic \emph{if and only if} it admits a genus $g \geq 2$ pencil with $m < 2g-2$ base points. This follows from Lemma~\ref{notrational} and Donaldson's result on the existence of Lefschetz pencils on symplectic $4$-manifolds (together with the fact there is a unique genus-$1$ pencil, which is on $CP$).

%\begin{remark}
%Our proof of the above lemma suggests a simple characterization of a symplectic $X$ in the homeomorphic class of a rational surface $\CP \# p \CPb$ with $p<9$ to be exotic: \emph{$X$ is exotic if and only if it contains an embedded symplectic surface $F$ with $g(F) \geq 2$ and $F^2 < 2g(F)-2$.} If $X$ is not rational (and it is clearly not ruled), by the work of Taubes, the canonical class $K_X$ can be represented by an embedded symplectic surface $F$. The adjunction equality implies $F^2=g-1$, whereas $F^2=c_1^2( \CP \# p \CPb) >0$ for $p<9$. Converse is shown in the proof of Lemma~\ref{notrational}. Such symplectic surfaces can be easily spotted in almost all earlier constructions of small exotic rational surfaces I know of.
%\end{remark}

Combining the lemmas~\ref{lemhomeo} and~\ref{notrational} we conclude:

\begin{theorem} \label{exoticrational}
$(X_i,f_i)$ are symplectic genus-$3$ Lefschetz pencils whose total spaces are homeomorphic but not diffeomorphic to $\CP \# (6+i) \CPb$ for $i=1,2, 3$.
\end{theorem}

\vspace{0.1in}
\begin{remark} \label{smallest}
These are the smallest genera examples one can find. Any Lefschetz pencil of genus $0$ or $1$ has a total space diffeomorphic to (a blow-up of) $\CPo \x \CPo$, $\CP \# \CPb$ or $E(1)=\CP \# 9 \CPb$ \cite{Kas, Moishezon}. After blowing-up all the base points of a genus-$2$ Lefschetz pencil we arrive at a genus-$2$ Lefschetz fibration on a non-minimal symplectic $(\W{X},\W{f})$. Now if $\W{X}$ is an exotic rational surface, it cannot have the rational homology type of a (blow-up of a) symplectic Calabi-Yau surface either. Thus there could be at most one base point, and as observed by Sato, $(\W{X},\W{f})$ has at most one reducible fiber; see Theorem~5-5(iii) and~Theorem 5-12(iii) in \cite{SatoKodaira}. By Lemmas~4 and~5 in \cite{BaykurKorkmaz}, there is no genus-$2$ Lefschetz fibration $(\W{X},\W{f})$ with at most one reducible fiber and say $\eu(\W{X}) \leq 13$. 

In fact, in \cite{BaykurKorkmaz}, we have constructed \emph{minimal} genus-$2$ Lefschetz fibrations whose total spaces are homeomorphic to rational surfaces $\CP \# \, p \CPb$, for $p=7,8,9$ \cite{BaykurKorkmaz}. These fibrations decompose as fiber sums, which implies their minimality \cite{Usher, BaykurPAMS}, and in turn why they are not diffeomorphic to blown-up rational surfaces.

The above observation generalizes to any genus $g \geq 3$ Lefschetz pencil $(X,f)$ whose total space is an exotic $\CP \# p \CPb$, for $p  < 9$, allowing us to conclude that $f$ has at most $2g-4$ base points. It follows from the following more general observation:  A collection of exceptional classes in the corresponding Lefschetz fibration $(\W{X}, \W{f})$ can be represented by disjoint multisections $S_j$, so each one intersects the regular fiber $F$ positively at least once. However, $(\sum S_j) \cdot F \geq  2g-3$ \, implies that $\kappa(X)= \kappa(\W{X} \leq 1$ by the work of Sato \cite{SatoKodaira}. Since $X$  is not rational (nor can possibly be ruled), we already know that $\kappa(X) \neq -\infty$. Further, $\kappa(X)=0$ or $1$ implies that the minimal model of $X$ has $c_1^2=0$, which is impossible for any $X$, as $c_1^2(X)=c_1^2(\CP \# p\,  \CPb)= 9-p >0$ and blow-down increases $c_1^2$. Hence $\kappa(X) = 2$ and has at most  $2g-4$ disjoint exceptional spheres. 

In particular we see that each $X_i$ in Theorem~\ref{exoticrational} is either minimal or it is only once blown-up of a minimal symplectic $4$-manifold.
\end{remark}

\begin{remark}
Using a variation of our construction with $3$ different embeddings of $P_1$ into $\Gamma_3$, we can get a genus-$3$ Lefschetz \emph{fibration} $(\W{X}_0, \W{f}_0)$ where $\W{X}_0$ is an exotic symplectic $\CP \# 7 \CPb$. I expect this monodromy lifts to $\Gamma_3^1$, and gives another pencil $(X_0, f_0)$ whose total space is an exotic $\CP \# 6 \CPb$. 
\end{remark}

\begin{remark} \label{distinctsymp}
The existence of exotic symplectic $4$-manifolds in the homeomorphism classes in Theorem~\ref{exoticrational}  was already established. The first example of an exotic symplectic $\CP \#  \, p \, \CPb$, for $p=9$, was the complex Dolgachev surface $E(1)_{2,3}$, as shown by Donaldson \cite{Donaldson87}. For $p=8$, the first example was Barlow's complex surface, as proved by Kotschick \cite{Kotschick}. For $p=7$, the first examples were constructed via generalized rational blowdowns by Jongil Park \cite{ParkJ}, which are symplectic by the work of Symington \cite{Symington}. Infinitely many distinct smooth structures in these homeomorphism classes were constructed using logarithmic transforms, knot surgeries and Luttinger surgeries \cite{FSKnotsurgery, Friedman, Szabo, FSDoublenode, ABP} (all of which are indeed instances of surgeries along tori \cite{BaykurSunukjian}.) 

Notably, only for $p=9$ it is known that there is an  \textit{infinite} family of pairwise non-diffeomorphic \textit{symplectic} $4$-manifolds in the homeomorphism class of $\CP \#  \, p \, \CPb$ (e.g. \cite{FSKnotsurgery}). Moreover, it was observed by Stipsicz and Szabo that Seiberg-Witten invariants cannot distinguish any two \textit{minimal} symplectic $4$-manifolds homeomorphic to $\CP \#  \, p \CPb$ for $p <9$ \cite{StipsiczSzabo}[Corollary~4.4]. It remains an open question whether there are two distinct symplectic $4$-manifolds in these homeomorphism classes \cite{Stern}[Problem~11]. It is thus desirable to have examples with more structure in order to address this intriguing question. For instance, one can use other $\phi$ to obtain simply-connected symplectic genus-$3$ pencils in the same homeomorphism classes, possibly not isomorphic to the ones in Theorem~\ref{exoticrational}. 
\end{remark}

%\newpage
\smallskip
%\vspace{0.2in}
% =========================================================
\subsection{An infinite family of genus-$3$ pencils with $c_1^2=0,1,2$ and $\chi_h=1$} \
% =========================================================

By varying our construction of the positive factorizations $W_i$, we can also produce symplectic genus-$3$ Lefschetz pencils with any prescribed abelian group of rank at most $2$. Let $M=(m_1,m_2)$, where $m_1$ and $m_2$ are nonnegative integers. For $i=1,2,3,$ let $(\widetilde{X}_{i,M},\widetilde{f}_{i,M})$ be the Lefschetz fibration prescribed by the positive factorization $W_i$ with $\phi_i= t^{-m_1}_{b_1} t_{a_2}^{m_2} $, and $({X}_{i,M},{f}_{i,M})$ be the corresponding pencil. Note that we still get a positive factorization $W_i$ since $\phi$ keeps $C$ and $C'$ fixed. 

We claim that $\pi_1(\widetilde{X}_{i,M}) \cong (\Z \, / {m_1 \, \Z}) \oplus (\Z \, / {m_2 \, \Z})$. To verify it, let's go back to our fundamental group calculation via the finite presentation which has the generators $a_j, b_j$ of $\pi_1(\Sigma_3)$, and relators induced by the Dehn twist curves in $W_i$.  The arguments we gave in Section~\ref{homeo} for the relators coming from the factors $P_1, P_2$ and $P'_2$ appearing in $W_i$ apply here as well. In particular,  $\pi_1(\widetilde{X}_{i,M}))$ is also a quotient of the abelian group $\Z^2$ generated by $a_2$ and $b_2$, and the following abelianized relations hold: 
\begin{align*} 
& \phantom{o} & a_3=-a_2 \, \text{ and }  \, b_3 =-b_2 &=b_1 \ &\text{for all $W_i$}, \\
& \phantom{o} & a_1+ 2a_2 +4 b_2 &= 0  \ \ \  &\text{for $W_1$,}  \\
& \phantom{o} & a_1+a_2 &= 0 \ \ \ &\text{for $W_2$ and $W_3$}. 
\end{align*}

Importantly, there are no other relations coming from the factors $P_1, P_2$ and $P'_2$: This is easy to see by abelianizing the relators $(~\ref{eqn:0})$--$(~\ref{eqn:B1})$, which include all the relators induced by the curves $x_1, x_2, x_3, B_0, B_1, B_2, B'_0, B'_1, B'_2$. Missing are the relators induced by the separating curves $d, e$ from $P_1$, the curves $A_0, A_1, A_2$ from $P_1$, and the curves  $A'_0, A'_1, A'_2$ from $P'_2$.  The first two are trivial in homology, so they have no contribution to the list of relators we already have. On the other hand, for each $j=0,1,2$,  $A_j$ is homologous to $B_j$, because  $[A_j]= [t_C(B_j)]=[B_j] + (C \cdot B_j) [C]$ by the Picard-Lefschetz formula, where $C$ is a separating cycle. Similarly each $A'_j$ is homologous to $B'_j$. Therefore the abelianized relations they induce are identical to those we already had from $B_j, B'_j$. 

For our conjugation $\phi= t_{b_1}^{-m_1} t_{a_2}^{m_2}$, we get the additional relators
\begin{align} 
& \phantom{o} & b_1 + m_2 a_2 +b_2 =0  &=0 \ \ \  \text{for all $W_i$,}  \label{abconj1} \\
& \phantom{o} & a_1 + m_1 b_1 + (2+4m_2) a_2 + 4 b_2 &=0 \ \ \  \text{for $W_1$ and $W_2$,}   \label{abconj2} \\
& \phantom{o} & a_1+ m_1 b_1+a_2 &= 0 \ \ \ \text{for $W_3$}. \label{abconj3}
\end{align}
Similar to before, $(\ref{abconj1})$ and $(\ref{abconj3})$, which hold for all $\pi_1(\widetilde{X}_{i,M})$, imply $m_2 a_2=0$. The remaining relators involved in a given $W_i$ then easily give $m_1 b_2=-m_1 b_1=0$, for each $i=1,2,3$. Hence, $\pi_1(X_{i,M})= \pi_1(\widetilde{X}_{i,M})= (\Z \, / {m_1 \, \Z}) \oplus (\Z \, / {m_2 \, \Z})$, as claimed.

Since the monodromy of each $(\widetilde{X}_{i,M}, \widetilde{f_{i,M}})$ is equivalent to that of $(\widetilde{X_i}, \widetilde{f_i})$ by a partial conjugation with $t_{b_1}^{-m_1+1} t_{a_2}^{m_2-1}$, they have the same Euler characteristic and signature. So $c_1^2(X_{i,M})=3-i$ and $\chi(X_{i,M})=1$, for $i=1,2,3$. For each $i$, we clearly have an infinite family of symplectic genus-$3$ pencils  among 
\[ \{ ({X}_{i,M},{f}_{i,M}) \ | \ M = (m_1, m_2) \in \N \times \N \ \} , \]
with pairwise homotopy inequivalent total spaces easily distinguished by $\pi_1(X_{i,M})$. By Parshin and Arakelov’s proofs of the Geometric Shafarevich Conjecture, there are finitely many holomorphic fibrations with fixed fiber genus $g\geq 2$ and degeneracy \cite{Parshin, Arakelov}. This implies that all but finitely many of these symplectic pencils are non-holomorphic. Furthermore, for $M=(m_1, 0)$ with $m_1 \geq 2$, we get an infinite subfamily of pencils, whose total spaces cannot be complex surfaces, since there is no complex surface $Z$ with $b_1(Z)=1$ and $b^+(Z)>0$ (e.g. cf. \cite{BaykurHolomorphic}[Lemma~2]). (Note that the same argument applies to the infinite family of symplectic Calabi-Yau surfaces we will construct in the next section.)

We have proved:

\begin{theorem} \label{nonholom}
There exists a symplectic genus-$3$ Lefschetz pencil $(X_{i,M}, f_{i,M})$ with \linebreak $c_1^2(X_{i,M})=3-i$, $\chi(X_{i,M})=1$, and $\pi_1(X_{i,M})=  (\Z \, / {m_1 \, \Z}) \oplus (\Z \, / {m_2 \, \Z})$, for any pair of non-negative integers $M=(m_1, m_2)$, and for each $i=1,2,3$. Infinitely many of these pencils have total spaces homotopy inequivalent to a complex surface.
\end{theorem}

\smallskip
\begin{remark}\label{nonholompencil}
There is only one example of a non-holomorphic genus-$3$ Lefschetz \emph{pencil} in the literature, which was given by Smith in \cite{SmithGenus3}[Theorem~1.3] without appealing to Donaldson's theorem.\footnote{In contrast, there are many non-holomorphic \emph{fibrations} obtained by twisted fiber sums; e.g. \cite{OzbagciStipsicz,Korkmaz,BaykurHolomorphic}.  However decomposable fibrations do not admit $(-1)$-sections ---which could be blown-down to produce pencils \cite{StipsiczFiberSum}.} (We learn from Naoyuki Monden that in a joint work with Hamada and Kobayashi they produce two more examples.) Theorem~\ref{nonholom} improves the situation, at least quantitatively.
\end{remark}

\begin{remark}  \label{pi1GenusInvariant}
Following the works of Donaldson \cite{Donaldson} and Gompf \cite{Gompf}, every finitely presented group is the fundamental group of a symplectic Lefschetz pencil. (Also see \cite{ABKP} and \cite{Korkmaz2}.) An invariant of finitely presented group $G$ is the $m_g(G)$, the minimal genus of a genus-$g$ pencil $(X,f)$ with $\pi_1(X)=G$, which is not as easy to calculate in general. However, Theorem~\ref{nonholom} shows that for $G \cong  (\Z \, / {m_1 \, \Z}) \oplus (\Z \, / {m_2 \, \Z})$, $m(G) \leq 3$. Well-known examples of genus-$1$ and genus-$2$ pencils show that $m_g(G)=1$ for $G=1$ (the pencil on $\CP$), and $m_g(G)=2$ for $G= \Z_2$ (the pencil on the Enriques surface) or $\Z \oplus \Z$ (e.g. Matsumoto's pencil on $S^2 \x T^2$). It seems likely that for all the other $G$ as above, $m_g(G)=3$.  
\end{remark}

%\newpage
\vspace{0.1in}
% =========================================================
\section{A new construction of symplectic Calabi-Yau surfaces with $b_1>0$}
% =========================================================

Here we will give a new construction of an infinite family of symplectic Calabi-Yau surfaces with $b_1=2$ and $3$, along with a symplectic Calabi-Yau surface with $b_1=4$, homeomorphic to the $4$-torus. This will follow from our construction of new positive factorization of boundary multi-twists in $\Gamma_3^4$ corresponding to symplectic genus-$3$ Lefschetz pencils. 

We find it more convenient for presentation purposes to discuss a special example first, and discuss the most general construction afterward. These will correspond to our examples with $b_1=4$, and then the others with $b_1=2$ and $3$.

\smallskip
% =========================================================
\subsection{A symplectic Calabi-Yau homeomorphic to the $4$-torus}  \label{Sec:4torus} \
% =========================================================

Our main building block will be the following relation in $\Gamma_2^2$ \cite{Hamada}:
\begin{eqnarray} \label{firstlift}
(t_{B_0} t_{B_1} t_{B_2} t_{C} )^2 &=& t_{\delta_1} t_{\delta_2}  \, ,
\end{eqnarray}
where $\delta_i$ are the boundary parallel curves, and the curves $B_i$ and $C$ are as shown on the left-hand side of Figure~\ref{Matsumoto}. 

\begin{figure}[p!]
 \centering
     \includegraphics[width=10cm]{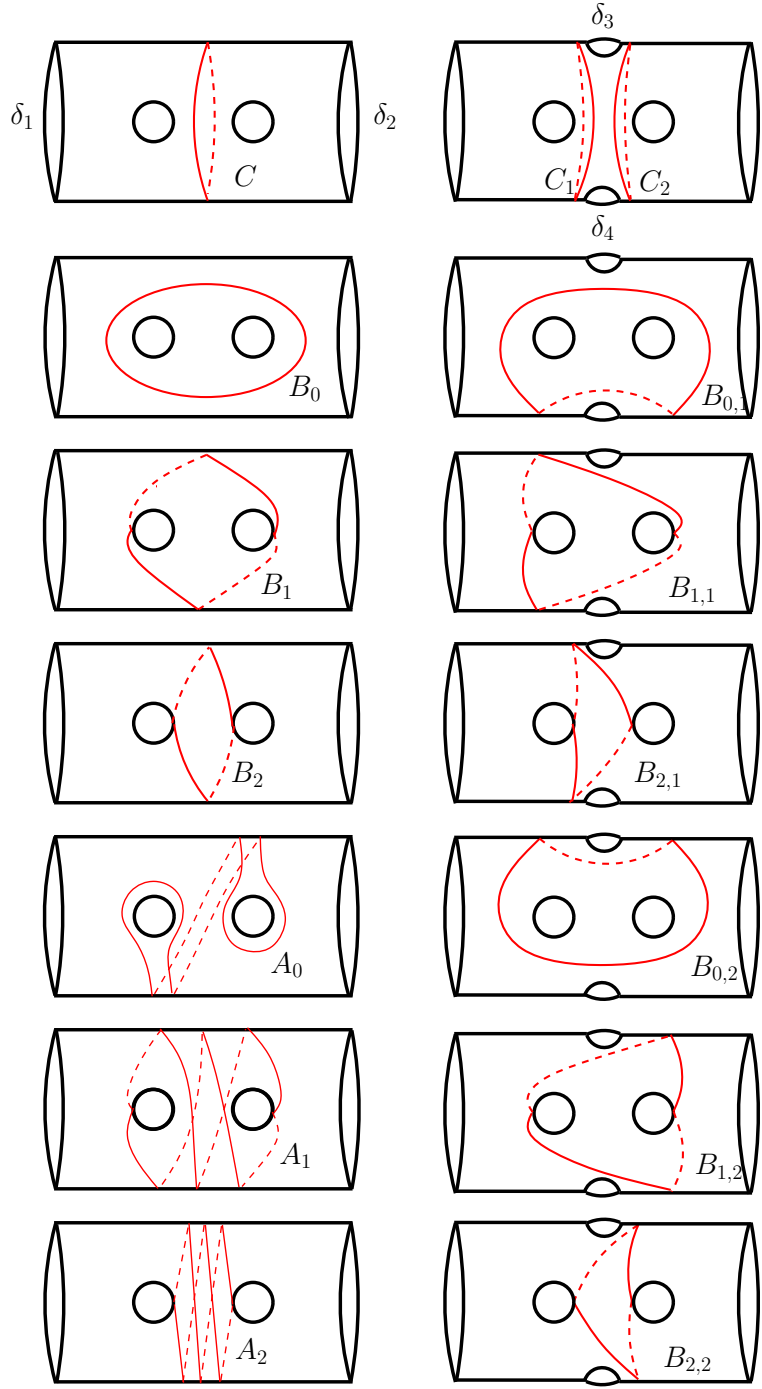}
\vspace{0.2in}
     \caption{Vanishing cycles $B_j, C, B_{j,i}$ and $C_i$ in Hamada's lift of Matsumoto's fibration. On the left are the curves of the positive factorization in $\Gamma_2^2$, along with the curves $A_j$ we get after Hurwitz moves. On the right are the curves of the further lift in $\Gamma_2^4$. }
     \label{Matsumoto}
\end{figure}

The relation $(\ref{firstlift})$ can be rewritten as
\begin{eqnarray*}
t_{B_0} t_{B_1} t_{B_2} t_{A_0} t_{A_1} t_{A_2} t_{C}^2 t_{\delta_1}^{-1} t_{\delta_2}^{-1} &=&  1  \, ,
\end{eqnarray*}
where each $A_j= t_C(B_j)$ for $j=0,1,2$  are given on the left-hand side Figure~\ref{Matsumoto}). We can embed this relation into $\Gamma_3$ by gluing along the two boundary components $\delta_1$ and $\delta_2$ so that $\delta_1=\delta_2$ is mapped to $C'$ and the curves $C, B_j, A_j$ are as shown in Figure~\ref{SCYcurves}, denoted by the same letters for simplicity. Thus the following holds in $\Gamma_3$: 
\begin{eqnarray} \label{embedSCY1}
t_{B_0} t_{B_1} t_{B_2} t_{A_0} t_{A_1} t_{A_2} t_{C}^2 t_{C'}^{-2} &=&  1  \, ,
\end{eqnarray}
or $P \, t_{C}^2 t_{C'}^{-2}=   1$, where $\underline{P= t_{B_0} t_{B_1} t_{B_2} t_{A_0} t_{A_1} t_{A_2}} $. Clearly $P, t_C$ and $t_{C'}$ all commute with each other. 

A similar embedding into $\Gamma_3$ can be given by mapping $\delta_1=\delta_2$ to $C$ instead, where the curves $C, B_j, A_j$ are now as shown in  Figure~\ref{SCYcurves}, all decorated by a prime notation this time (so $C$ indeed maps to $C'$ above). We thus get another relation in $\Gamma_3$ of the form:
\begin{eqnarray} \label{embedSCY2}
t_{B'_0} t_{B'_1} t_{B'_2} t_{A'_0} t_{A'_1} t_{A'_2} t_{C'}^2 t_{C}^{-2} &=&  1  \, ,
\end{eqnarray}
or $P' \, t_{C'}^2 t_{C}^{-2}=   1$, where $\underline{P'= t_{B'_0} t_{B'_1} t_{B'_2} t_{A'_0} t_{A'_1} t_{A'_2}} $. Similarly, $P', t_C$ and $t_{C'}$ all commute with each other. 

Combining the above, we get a positive factorization $\underline{W=P P'}$ in $\Gamma_3$, since
\[  P P' =  P \, t_{C}^2 t_{C'}^{-2} \, P' \, t_{C'}^2 t_{C}^{-2} = 1 \, . \]
Let $(\widetilde{X}, \widetilde{f})$ be the genus-$3$ Lefschetz fibration prescribed by $W$. We can equip $(\W{X}, \W{f})$ with a Gompf-Thurston symplectic form $\W{\omega}$, in particular $\W{X}$ is a symplectic $4$-manifold. We will show that $\widetilde{X}$ is the $4$ times blow-up of a symplectic Calabi-Yau surface $X$ homeomorphic to the $4$-torus. We will indeed show that it is a blow-up of a symplectic genus-$3$ pencil $(X,f)$, which we will first argue implicitly below, and then explicitly in the next section.

\begin{lemma}\label{SCYpencil}
Let $(\W{X},\W{f})$ be a genus $g \geq 2$ Lefschetz fibration and $c_1^2(\W{X})=2-2g$. If $\W{X}$ is not a rational or ruled surface, then $(\W{X},\W{f})$ is a blow-up of a symplectic Lefschetz pencil $(X,f)$, where $X$ is a (minimal) symplectic Calabi-Yau surface.
\end{lemma}

\begin{proof}
We can equip $(\W{X},\W{f})$ with a Gompf-Thurston symplectic form $\W{\omega}$. By the work of Taubes \cite{Taubes, Taubes2}, any \emph{minimal} symplectic $4$-manifold has non-negative $c_1^2$. So $\W{X} \cong X_0 \# m \CPb$, where $X_0$ is a minimal model for $\W{X}$ and $m \geq 2g-2$, which means that $\W{X}$ contains at least $2g-2$ disjoint exceptional spheres $S_j$, $j=1, \ldots, 2g-2$. There is an $\W{\omega}$-compatible almost complex structure $J$ with respect to which both a regular fiber $F$ of $\W{f}$ and all $S_j$ are $J$-holomorphic \cite{SatoKodaira, BaykurPAMS}, which implies that $(\sum_{j=1}^{2g-2} S_j) \cdot F \geq 2g-2$. By the work of Sato \cite{SatoKodaira} (also see \cite{BaykurHayano}), this is only possible if the symplectic Kodaira dimension $\kappa(\W{X}) \leq 0$. As we assumed $\W{X}$ is not a rational or a ruled surface, we have $\kappa(\W{X}) \neq - \infty$ \cite{LiSCY}, leaving the only possibility as $\kappa(\W{X})=0$ and also that there are no other disjoint exceptional spheres than $S_j$ in $\W{X}$. Thus $(\sum_{j=1}^{2g-2} S_j) \cdot F = 2g-2$, so one can blow-down these exceptional spheres (which can be assumed to be symplectic after deforming $\W{\omega}$ if needed \cite{GompfStipsicz}) to arrive at a symplectic genus-$3$ Lefschetz pencil $(X,f)$, where $X= X_0$ is a minimal symplectic Calabi-Yau surface.
\end{proof}

We show that $\W{X}$ satisfies the hypotheses of this lemma. 

The Euler characteristic of $\W{X}$ is given by $\eu(\W{X})= 4-4 \cdot 3 + \ell(W)= 4$, where $\ell(W)=12$ is the number of Dehn twists in the factorization $W$. 

The signature of $\W{X}$ can be calculated as $\sigma(\W{X})=-4$  using the local signatures and  Ozbagci's algorithm \cite{Ozbagci}. Since $(\W{X}, \W{f})$ is not hyperelliptic, we do not have the luxury of employing Endo's closed formula for a direct calculation, however the following observation allows us to deduce the signature from a calculation in $\Gamma_2$, where any factorization is hyperelliptic: We have obtained $W$ from the factorization which had two pairs of $t_C$ and $t_{C'}$ with opposite signs. This is an \emph{achiral} genus-$3$ Lefschetz fibration, which is a fiber sum of two achiral fibrations prescribed by the factorizations $P \, t_{C}^2 t_{C'}^{-2}$ and $P' \, t_{C'}^2 t_{C}^{-2}$. Each one is a \emph{section sum} of an achiral genus-$2$ Lefschetz fibration and a trivial torus fibration along a section of self-intersection zero, which, by  Novikov additivity, has signature equals to that of the achiral genus-$2$ Lefschetz fibration. By the hyperelliptic signature formula for genus-$2$ fibrations \cite{Matsumoto, Endo}, the signature of the latter is $-2$. Invoking Novikov additivity again, we conclude that the decomposable achiral genus-$3$ fibration has signature $-4$, the sum of the signature of its summands. Lastly, we claim that canceling the pairs $t^{\pm 2}_C$ and $t^{\pm 2}_{C'}$ does not change the signature. This can be easily seen by splitting the fibration as  $t^2_C t^{-2}_C$ (or $t^2_{C'} t^{-2}_{C'}$) and the rest. The piece with monodromy factorization $t^2_C t^{-2}_C$ has zero signature, which is easily calculated from the induced handle decomposition, or by observing that it can be further split into two identical fibrations whose total spaces have opposite orientations (and thus opposite signatures). By Novikov additivity once again, the rest of the fibration has the same signature as the total signature of the achiral fibration. So $\sigma(\W{X})=-4$.

It follows that $c_1^2(\W{X})= 2 \eu(\W{X}) + 3\sigma(\W{X})=-4$. Assuming $\pi_1(\W{X})=\Z^4$ for the moment, let us complete our arguments regarding the homeomorphism type of the minimal model $X$ of $\W{X}$. By Lemma~\ref{SCYpencil} $(\W{X}, \W{f})$ is a blow-up of a symplectic genus-$3$ Lefschetz pencil $(X,f)$, where $X$ is a symplectic Calabi-Yau surface. Since $X$ is a symplectic Calabi-Yau surface with $b_1(X)>0$, it has the rational homology type of a torus bundle over a torus \cite{Bauer, LiSCY}. Moreover, since $\pi_1(X)=\Z^4$ is a virtually poly-$\Z$ group, the Borel conjecture holds in this case by the work of Farrell and Jones \cite{FJ90}, and as observed by Friedl and Vidussi, this implies that a symplectic Calabi Yau surface with $\pi_1(\Z^4)$ is unique up to homeomorphism \cite{FriedlVidussi}. Hence $X$ is homeomorphic to the well-known symplectic Calabi-Yau surface, the $4$-torus. 

We now calculate $\pi_1(X)=\pi_1(\W{X})$, which will constitute the rest of this section. Using the standard handlebody decomposition for $(\W{X}, \W{f})$, we can obtain a finite presentation for $\pi_1(\W{X}) \cong \pi_1(\Sigma_3) \, / \, N$, where $N$ denotes the subgroup of $\pi_1(\Sigma_3)$ generated normally by the vanishing cycles of $\W{f}$. For $\{a_j, b_j\}$ standard generators of $\pi_1(\Sigma_g)$ as given in Figure~\ref{pi1generators}, we get
\[ \pi_1(X) \cong \, \langle \, a_1, b_1, a_2, b_2, a_3, b_3 \, | \, [a_1, b_1] [a_2, b_2] [a_3, b_3] , R_1, \ldots, R_{12} \, \rangle \, , \]
where each $R_i$ is a relation obtained by expressing a corresponding vanishing cycle (oriented arbitrarily) in  $\{a_j, b_j\}$. Recall that $(\W{X}, \W{f})$ has a section (in fact $4$ \linebreak $(-1)$-sections), so the above presentation is complete, there is no other relation.

\begin{figure}[ht!]
 \centering
     \includegraphics[width=5cm]{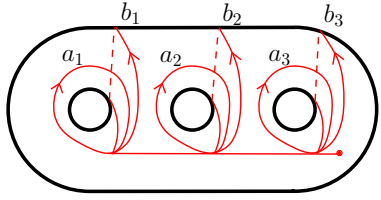}
     \caption{Generators $a_j, b_j$ of $\pi_1(\Sigma_3)$}
     \label{pi1generators}
\end{figure}

We have the following relations, the ones after the first one coming from $\Sigma_3$ all induced by the vanishing cycles $B_0$, $B_1$, $B_2$, $A_0$, $A_1$, $A_2$, $B'_0$, $B'_1$, $B'_2$, $A'_0$, $A'_1$, $A'_2$ (which are all given in Figure~\ref{SCYcurves}, in the given order:
 \begin{eqnarray}
& \phantom{o} &  [a_1, b_1][a_2, b_2][a_3,b_3]=1 ,      \label{0} \\  
& \phantom{o} &  a_1 a_3=1,      \label{B0} \\  
& \phantom{o} &  a_1  \bar{b}_1   a_2 b_2 \bar{a}_2 a_3 \bar{b}_3=1,      \label{B1} \\  
& \phantom{o} &   \bar{b}_1 a_2 b_2 \bar{a}_2 \bar{b}_3 =1, \label{B2} \\  
& \phantom{o} &  a_1 [b_3, a_3] \, b_2 a_3 \bar{b}_2 \, [a_3, b_3]=1,      \label{A0} \\  
& \phantom{o} &   a_3 \bar{b}_3 \bar{b}_2 [a_3, b_3] \, a^2_1 \bar{b}_1 \bar{a}_1 [b_3, a_3] \, b_2 [b_3, a_3] \, b_2=1,      \label{A1} \\  
& \phantom{o} &  a_1  \bar{b}_1 \bar{a}_1 [b_3, a_3] \, b_2 [b_3, a_3] \, b_2 \bar{b}_3 \bar{b}_2 [a_3, b_3]   =1, \label{A2} \\  
& \phantom{o} &  \bar{a}_2 a_1 a_2 a_3=1,      \label{B'0} \\  
& \phantom{o} &  a_1 \bar{b}_1 a_2 a^2_3 \, \bar{b}_3 \bar{a}_3 b_2 \bar{a}_2 =1,      \label{B'1} \\  
& \phantom{o} &   b_1 a_2 \bar{b}_2 a_3 b_3 \bar{a}_3 \bar{a}_2 =1, \label{B'2} \\  
& \phantom{o} &  a_1 a_2 \bar{b}_2 a_3 b_2 \bar{a}_2=1,      \label{A'0} \\  
& \phantom{o} & a_1 \bar{b}_1 a_2 \bar{b}_2 a^2_3 \, \bar{b}_3 \bar{a}_3 b^2_2 \, \bar{a}_2=1,      \label{A'1} \\  
& \phantom{o} &  \bar{b}_1 a_2 \bar{b}_2 a_3 \bar{b}_3 \bar{a}_3 b^2_2 \, \bar{a}_2 =1, \label{A'2} 
\end{eqnarray}
Here $\bar{x}$ denotes the inverse of the element $x$ in the group. 

\begin{figure}[p!]
 \centering
     \includegraphics[height=18.5cm]{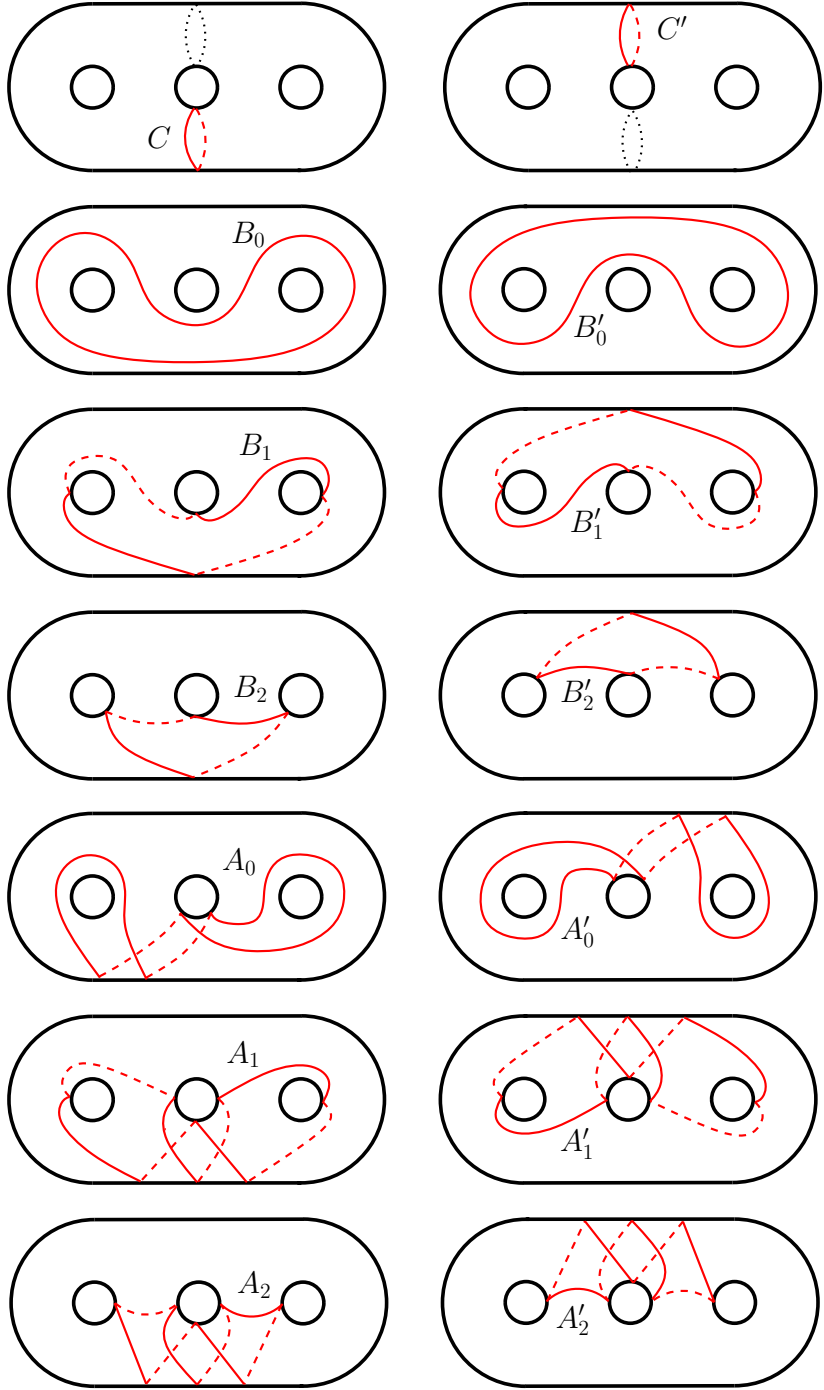}
\smallskip
     \caption{Vanishing cycles $B_j, A_j, B'_j, A'_j$ of the genus-$3$ fibration $(X,f)$. On the left are the curves coming from the factorization $P$ and on the right are those coming from the factorization $P'$, which correspond to the two different embeddings of the factorization in $\Gamma_2^2$ into $\Gamma_3$. (Dotted lines drawn on the surfaces on the top are the images of $\delta_1=\delta_2$ under these embeddings.)}
     \label{SCYcurves}
\end{figure}

When abelianized, the relations coming from each one of the triples $\{B_0, B_1, B_2\}$, $\{A_0, A_1, A_2\}$, $\{B'_0, B'_1, B'_2\}$, $\{A'_0, A'_1, A'_2\}$ give the same three relations 
\begin{eqnarray*}
& \phantom{o} &  a_1 +  a_3= 0,     \\
& \phantom{o} &  a_1  -b_1 + b_2 + a_3 - b_3= 0,    \\
& \phantom{o} & -b_1 +b_2 -b_3 =0, 
\end{eqnarray*}
where we identified the images of the generators with the same letters.  Any $2$ of these relations imply the other. So we can eliminate $a_1=-a_3$ and $b_1=b_2-b_3$, and get a free abelian group of rank $4$ generated by $a_2, b_2, a_3$ and $b_3$. 

Now, going back to the presentation we have above for $\pi_1(\W{X})$, we easily see that it is also generated by $a_2, b_2, a_3, b_3$, since $a_1= \bar{a}_3$ by $(\ref{B0})$ and $b_1 = a_2 b_2 \bar{a}_2 \bar{b}_3$ by $(\ref{B2})$. To prove $\pi_1(\W{X})=\Z^4$, it therefore suffices to show that $a_2, b_2, a_3, b_3$ all commute with each other. 

Replacing $a_1$ with $\bar{a}_3$ in $(\ref{B'0})$ gives $[a_3, a_2]=1$. From  $(\ref{B2})$ we have $\bar{b}_1 a_2 b_2 \bar{a}_2 = b_3$. Substituting this in $(\ref{B1})$, and replacing $a_1$ with $\bar{a}_3$, we get $[a_3, b_3]=1$. With  $a_1=\bar{a}_3$ and $[a_3, b_3]=1$, the relation $(\ref{A0})$ simplifies to $[a_3, b_2]=1$. So $a_3$ commute with each one of $a_2, b_2, b_3$. 
We can rewrite $(\ref{A'0})$ as $a_1 a_2 \bar{b}_2 a_3 b_2 \bar{a}_2= a_2 \bar{b}_2 a_3 b_2 \bar{a}_2 =1$, since  $a_1=\bar{a}_3$ and $a_3$ commutes with $a_2$ and $\bar{b}_2$. So
$[a_2, b_2]=1$. It implies that  $b_1 = a_2 b_2 \bar{a}_2 \bar{b}_3 = b_1 \bar{b}_3$. Note that $[a_3, b_3]=1$ means $a_3 b_3 \bar{a}_3 = b_3$. Using these last two identities and $(\ref{B'2})$ we derive $b_2 \bar{b}_3 a_2 \bar{b}_2 b_3 \bar{a}_2 =1$. So $[b_2, \bar{b}_3 a_2] =[a_2, \bar{b}_2 b_3]=1$. These commuting relations, together with $[a_2, b_2]=1$, give us the last two commuting relations we need: $[b_2, b_3]=[a_2, b_3]=1$. 

Hence $\pi_1(\W{X}) \cong \Z^4$, generated by $a_2, b_2, a_3, b_3$. This completes our claim that $(X,f)$ is a symplectic genus-$3$ Lefschetz pencil, where $X$ is a symplectic Calabi-Yau surface homeomorphic to the $4$-torus.

\smallskip
% =========================================================
\subsection{Symplectic Calabi-Yau surfaces with $b_1=2$ and $3$ via pencils}  \
% =========================================================

We will now generalize the construction in the previous section, and construct an infinite family of examples of genus-$3$ Lefschetz pencils on symplectic Calabi-Yau surfaces with $b_1=2$ and $3$. We will produce these pencils explicitly via positive factorizations of the boundary multi-twist $\Delta= t_{\partial_1} t_{\partial_2} t_{\partial_3} t_{\partial_4}$ in $\Gamma_2^4$, where $\partial_i$ are boundary parallel curves for distinct boundary components of $\Sigma_3^4$. For this, we will use a further lift of the monodromy factorization of Matsumoto's genus-$2$ fibration to $\Gamma_2^4$ given by Hamada, which will allow us to obtain a nice and symmetric presentation at the end. 

The factorization $(\ref{firstlift})$ has a further lift to $\Gamma_2^4$ \cite{Hamada} of the form:
\begin{eqnarray} \label{secondlift}
t_{B_{0,1}} t_{B_{1,1}} t_{B_{2,1}} t_{C_1}  \, t_{B_{0,2}} t_{B_{1,2}} t_{B_{2,2}} t_{C_2}  &=& t_{\delta_1} t_{\delta_2} t_{\delta_3} t_{\delta_4}  \, ,
\end{eqnarray}
where again $\delta_i$ are the boundary parallel curves, and the curves $B_{j,i}, C_i$ are given on the right-hand side of Figure~\ref{Matsumoto}. When the two boundary components $\delta_3, \delta_4$, drawn in the middle of $\Sigma_2^4$ in Figure~\ref{Matsumoto}, are capped off, the curves  $B_{j,i}$ descend to $B_j$ and $C_i$ to $C$ for each  $j=0,1,2$ and $i=1,2$. 

\begin{figure}[ht!]
 \centering
     \includegraphics[width=5cm]{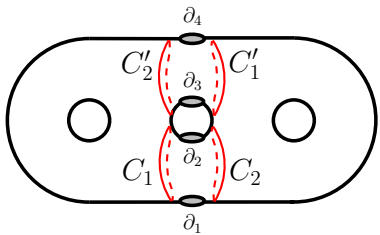}

     \caption{Curves involved in our embeddings of $\partial \Sigma_4^2$ into $\partial \Sigma_3^4$.}
     \label{SCYgluing}
\end{figure}

The relation $(\ref{secondlift})$ can be rewritten as the following relation in $\Gamma_2^4$:
\begin{eqnarray*}
t_{B_{0,1}} t_{B_{1,1}} t_{B_{2,1}} \, t_{A_{0,2}} t_{A_{1,2}} t_{A_{2,2}} t_{C_1} t_{C_2} t^{-1}_{\delta_1} t^{-1}_{\delta_2}    &=&  t_{\delta_3} t_{\delta_4}  \, ,
\end{eqnarray*}
where each $A_{j,2}= t_{C_1}(B_{j,2})$ for $j=0,1,2$. Embed this configuration of curves on $\Sigma_2^4$ into $\Sigma_3^4$ by gluing a cylinder with $2$ holes along two of the boundary components of $\Sigma_3^4$. Choose the embedding so that $\delta_1$ maps to $C'_2$, $\delta_2$ to $C'_1$, $\delta_3$ to $\partial_2$ and $\delta_4$ to $\partial_1$. (See Figure~\ref{SCYgluing}.) We then map the interior of $\Sigma_2^4$ so that the curves $B_{j,1}$ and $A_{j,2}$ all map to the curves $B_j$ and $A_j$ in Figure~\ref{SCYcurves} when the boundary components $\partial_3$ and $\partial_4$ are capped off. (That is, we map the interior so that when $\partial \Sigma_3^4$ is capped off, we get the first embedding we had in the previous section.)  Thus the following holds in $\Gamma_3^4$: 
\begin{eqnarray} \label{embedSCY3}
t_{B_{0,1}} t_{B_{1,1}} t_{B_{2,1}} \, t_{A_{0,2}} t_{A_{1,2}} t_{A_{2,2}} t_{C_1} t_{C_2} t^{-1}_{C'_2} t^{-1}_{C'_1}    &=&  t_{\partial_1} t_{\partial_2}  \, ,
\end{eqnarray}
or $\W{P} \, t_{C_1} t_{C_2} t^{-1}_{C'_2} t^{-1}_{C'_1} =   1$, where $\underline{\W{P}=t_{B_{0,1}} t_{B_{1,1}} t_{B_{2,1}} \, t_{A_{0,2}} t_{A_{1,2}} t_{A_{2,2}}}$. 

A similar embedding into $\Gamma_3^4$ can be given by mapping $\delta_1$ to $C_2$, $\delta_2 \to C_1$, $\delta_3 \to \partial_3$ and $\delta_4 \to \partial_4$ (see Figure~\ref{SCYgluing}, where the interior is mapped in a similar fashion as above so as to get the curves $B'_j$ and $A'_j$ of Figure~\ref{SCYcurves}. We get another relation in $\Gamma_3^4$ of the form:
\begin{eqnarray} \label{embedSCY2}
t_{B'_{0,1}} t_{B'_{1,1}} t_{B'_{2,1}} \, t_{A'_{0,2}} t_{A'_{1,2}} t_{A'_{2,2}}   t_{C'_1} t_{C'_2} t^{-1}_{C_2} t^{-1}_{C_1}    &=&  t_{\partial_3} t_{\partial_4}   \, ,
\end{eqnarray}
or $\W{P}' \,  t_{C'_1} t_{C'_2} t^{-1}_{C_2} t^{-1}_{C_1}  =  1$, where $\underline{\W{P}'=t_{B'_{0,1}} t_{B'_{1,1}} t_{B'_{2,1}} \, t_{A'_{0,2}} t_{A'_{1,2}} t_{A'_{2,2}}}$. 

Now let $\phi$ be any mapping class in $\Gamma_3^4$ which fix the curves $C_1, C_2, C'_1, C'_2$. Then the product of $\W{P}^{\phi}$ and $\W{P}'$ then give:
\begin{align*}
& \phantom{o} & \W{P}^{\phi} \W{P}'  &= \W{P}^{\phi} \,  t_{C_1} t_{C_2} t^{-1}_{C'_2} t^{-1}_{C'_1} \ \W{P}'   t_{C'_1} t_{C'_2} t^{-1}_{C_2} t^{-1}_{C_1} \\
& \phantom{o} & &=  (\W{P} \,  t_{C_1} t_{C_2} t^{-1}_{C'_2} t^{-1}_{C'_1})^{\phi} \ \W{P}'   t_{C'_1} t_{C'_2} t^{-1}_{C_2} t^{-1}_{C_1} \\
& \phantom{o} & &= t_{\partial_1} t_{\partial_2} t_{\partial_3} t_{\partial_4}  \, .
\end{align*}
In the first equality we used the commutativity of disjoint Dehn twists $t_{C_1}$, $t_{C_2}$, $t_{C'_1}$, $t_{C'_2}$ and that they all commute with  $\W{P}$ (and $\W{P}'$). The second equality comes from our choice of $\phi$ so that $\phi$ commutes with Dehn twists along $C_1$, $C_2$, $C'_1$, $C'_2$.  Therefore $\underline{W_{\phi}= \W{P}^\phi \W{P}'}$ is a positive factorization of the boundary multi-twist $\Delta=t_{\partial_1} t_{\partial_2} t_{\partial_3} t_{\partial_4}$ in $\Gamma_3^4$ for any $\phi$ in $\Gamma_3^4$ fixing the curves $C_1, C_2, C'_1, C'_2$ on $\Sigma_3^4$.

Letting $(X_{\phi}, f_{\phi})$ denote the symplectic genus-$3$ Lefschetz pencil corresponding to the positive factorization $W_{\phi}$, we can now state our general result as:

\begin{theorem} \label{SCYfamily}
Each $(X_{\phi}, f_{\phi})$ is a symplectic genus-$3$ Lefschetz pencil on a minimal symplectic Calabi-Yau surface, where $X$ is homeomorphic to the $4$-torus for $\phi=1$. The family of genus-$3$ Lefschetz pencils
\[ \{ (X_{\phi}, f_{\phi}) \ | \ \phi \in \Gamma_3^4 \text{ and keeps } C_1, C_2, C'_1, C'_2 \text{ fixed} \, \}\]
contains a subfamily 
\[ \{ (X_M, f_M) \ |  \ H_1(X_M) \cong \Z^2 \oplus (\Z \, / {m_1 \, \Z}) \oplus (\Z \, / {m_2 \, \Z}) \, , \text{ for } \, M=(m_1, m_2) \in \N \x \N,    \}. \] 
\end{theorem}

We will appeal to the following in the proof of our theorem:

\begin{lemma} \label{ruledno}
The manifolds $S^2 \x T^2$ and $S^2 \W{\x} T^2$ do not admit any genus-$3$ Lefschetz pencils with $4$ base points. 
\end{lemma}

\begin{proof}
We will show that neither one of these ruled surfaces contains an embedded symplectic surface $F$ of genus $3$ and self-intersection $4$, whereas the fiber of a pencil as in the statement of lemma would give such an $F$.  

For $X= S^2 \x T^2$, $H_2(X) \cong \Z^2$ is generated by $S= S^2 \x \{ pt\}$ and $T= \{ pt\} \x T^2$, where $S \cdot S=0$, $T \cdot T=0$, and $S \cdot T=1$. Now let $F= a S + bT$ be a symplectic curve of genus $3$ with $F^2=4$. So $4= F^2= 2ab$, which implies $ab=2$. Since there is a unique symplectic structure on a minimal ruled surface up to deformations and symplectomorphisms \cite{LiLiu}, we can apply the adjunction formula and derive 
\[ 4= \eu(F)= F^2 + K_{X} \cdot F = 4 + (-2T) \cdot (aS+bT) = 4- 2a \, , \]
and get $a=0$. It however contradicts with $ab=2$ above, so there is no such $F$.

For $X=S^2 \W{\x} T^2$, $H_2(X) \cong \Z^2$ is generated by the fiber $S$ and section $T$ of the degree-$1$ ruling on $X$, where $S \cdot S=0$, $T \cdot T=1$, and $S \cdot T=1$. Let $F= a S + bT$ be a symplectic curve of genus $3$ with $F^2=4$. So $4= F^2= 2ab +b^2$, which implies $b(2a+b)=4$. By the same argument as above, we can apply the adjunction to derive
\[ 4= \eu(F)= F^2 + K_{X} \cdot F = 4 + (S-2T) \cdot (aS+bT) = 4 -2a-b \, , \]
and get $2a+b=0$, which contradicts with $b(2a+b)=4$. So there is no such $F$ in this case either. 
\end{proof}

We can now prove our theorem:

\begin{proof}[Proof of Theorem~\ref{SCYfamily}]
For $\phi=1$ we get the pencil $(X,f)$ constructed in the previous section, now with an explicit lift of its monodromy to $\Gamma_3^4$ as a positive factorization $\W{P} \W{P}'= \Delta$. Since any $(X_{\phi}, f_{\phi})$ can be derived from $(X,f)$ by a partial conjugation along the first factor $\W{P}$ in this factorization, $\eu(X_{\phi})= \eu(X)=4$, $\sigma(X_{\phi})=\sigma(X)=-4$, and 
$c_1^2(X_{\phi})= c_1^2(X)=-4$. By Lemma~\ref{SCYpencil}, we conclude that $X_{\phi}$ is a minimal symplectic Calabi-Yau surface, provided it is not rational or ruled.

From $2 - 2 b_1(X_{\phi})+ b^+(X_{\phi}) + b^-(X_{\phi}) = \eu(X_{\phi}) = 4$ and  $b^+(X_{\phi})-b^-(X_{\phi}) = \sigma(X_{\phi}) = -4$ we see that $b^+(X_{\phi})=1$ only if $b_1(X_{\phi})=2$. So the only rational or ruled surface with these characteristic numbers can be $S^2 \x T^2$ or $S^2 \W{\x} T^2$. However, by the previous lemma, neither one of these surfaces admit a genus-$3$ pencil, so any $X_{\phi}$ is a symplectic Calabi-Yau surface.

To prove the second part of the theorem, consider genus-$3$ pencils $(X_{\phi}, f_{\phi})$ with $\phi= t_{b_1}^{-m_1} t_{a_3}^{m_2}$ (where $b_1$ and $a_3$ are as in Figure~\ref{pi1generators}). Note that $b_1$ and $a_3$ are disjoint from $C_1, C_2, C'_1, C'_2$, so $\phi$ fixes all. For $m= (m_1, m_2) $ any pair of non-negative integers, we denote this pencil by $(X_m, f_m)$, and its monodromy factorization by $W_m= \W{P}^\phi \W{P}'$, $\phi= t_{b_1}^{-m_1} t_{a_3}^{m_2}$.

As we observed in the previous section, every triple of vanishing cycles $\{B_0, B_1, B_2\}$, $\{A_0, A_1, A_2\}$, $\{B'_0, B'_1, B'_2\}$, $\{A'_0, A'_1, A'_2\}$ give the same three relations in $H_1$:
\begin{eqnarray}
& \phantom{o} &  a_1 +  a_3= 0,   \label{varyingab1}   \\
& \phantom{o} &  a_1  -b_1 + b_2 + a_3 - b_3= 0,     \label{varyingab2}  \\
& \phantom{o} & b_1 -b_2 +b_3 =0,  \label{varyingab3}
\end{eqnarray}
where the first and the third relations imply the second. We can draw two conclusions: First, the vanishing cycles coming from the non-conjugated factor $\W{P}'$ induce exactly these relations in $H_1(X_m)$. Second, the vanishing cycles coming from the conjugated factor $\W{P}^{\phi}$ induce the following relations we can easily calculate using the Picard-Lefschetz formula:
\begin{eqnarray}
& \phantom{o} &  a_1 + m_1 b_1 +  a_3= 0,  \label{varyingab4}  \\
& \phantom{o} &  a_1 + m_1 b_1  -b_1 + b_2 + a_3 - m_2 a_3 - b_3= 0,    \label{varyingab5}  \\
& \phantom{o} & b_1 -b_2 + m_2 a_3+ b_3  =0 \, .  \label{varyingab6}
\end{eqnarray}
We easily see that $(\ref{varyingab1})$ and $(\ref{varyingab4})$ imply $m_1 b_1=0$, and $(\ref{varyingab3})$ and $(\ref{varyingab6})$ imply $m_2 a_3=0$, whereas the others are already implied by these relations. From $a_3= -a_1$ and $b_3= b_2-b_1$, we conclude that $H_1(X_m)$ is generated by $a_1, b_1, a_2, b_2$ with $m_1 b_1=0$ and $m_2 a_1 = 0$. Hence
$H_1(X_m) \cong \Z^2 \oplus (\Z \, / {m_1 \, \Z}) \oplus (\Z \, / {m_2 \, \Z})$, as promised.
\end{proof} 

\smallskip
\begin{remark} \label{CompareSmith}
The very first question here is whether or not every $X_{\phi}$ is a torus bundle over a torus, as they are commonly conjectured to exhaust the list of all symplectic Calabi-Yau surfaces with $b_1 >0$. If for any $\phi$, $\pi_1(X_{\phi})$ is not a $4$-dimensional solvmanifold group \cite{Hillman, FriedlVidussi}, this would imply that $X_{\phi}$ is \emph{not} a torus bundle over a torus, and it is a new symplectic Calabi-Yau surface.  As our arguments in the proof of Theorem~\ref{SCYfamily} show, if any partial conjugation along any Hurwitz equivalent factorization to $W_{\phi}$ results in a pencil with a fundamental group which is not a solvmanifold group, we can arrive at a similar conclusion as well. I do not know for the moment if either one of these can be realized or ruled out for all possible $\phi$.

It is worth comparing our pencils with those on torus bundles over tori. In \cite{SmithTorus} Ivan Smith constructs genus-$3$ pencils on torus bundles \emph{admitting sections} (not all do), by generalizing the algebraic geometric construction of holomorphic genus-$3$ pencils on abelian surfaces. A natural question is whether or not our examples overlap with Smith's. (Well, if they do, the explicit factorizations $W_{\phi}$ can be seen to prescribe very nice handle decompositions for such torus bundles.)

Note that that any torus bundle over a torus with a section $S$ would admit a second disjoint section as well; for any section of a surface bundle over a torus has self-intersection zero \cite{BaykurKorkmazMonden} and can be pushed-off of itself. So any genus-$3$ Lefschetz fibration of Smith would be determined by a pair $\phi_1, \phi_2 \in \Gamma_1^2$ \emph{subject to} $[\phi_1, \phi_2]=1$.  

Now for a comparison, let us note that the stabilizer subgroup of $\Gamma_3^4$ which fix $C_1$, $C_2$, $C'_1$, $C'_2$, is generated by Dehn twists along curves which do not intersect the stabilized collection, and the obvious involution swapping the genus one surfaces bounded by $C_1, C'_2$ and by $C'_1, C_2$. (This involution maps each $A_j, B_j, A'_j, B'_j$ to itself!)  So any $\phi$ in Theorem~\ref{SCYfamily} can be expressed as $\phi= t_{z_k} \cdots t_{z_1}$, where each $z_i$, for $i=1, \ldots, k$, is either contained in the subsurface $F_1 \cong \Gamma_1^2$ in $\Gamma_3^4$ bounded by $C_1$ and $C'_2$, or  $F_2 \cong \Gamma_1^2$ bounded by $C_2$ and $C'_1$. Clearly these have disjoint supports, so we can rearrange the indices so that $\phi = \phi_1 \phi_2$, where each $\phi_i$ is compactly supported on $F_i$. Hence we see that $(X_{\phi}, f_{\phi})$ is determined by a pair of mapping classes $\phi_1, \phi_2 \in \Gamma_1^2$, \emph{but with no relation to each other whatsoever!}. 
\end{remark}

\begin{remark} \label{CompareHoLi}
As observed in the previous remark, for each $\phi$ in Theorem~\ref{SCYfamily}, we have factorization $\phi= t_{z_k} \cdots t_{z_1}$, where $z_i$ are disjoint from $C_1, C_2, C'_1, C'_2$. So
 \[ W_{\phi}= \W{P}^{\phi} \W{P}' = ( \phi \W{P} \phi^{-1}) \, \W{P}' = (t_{z_k} ( \cdots (t_{z_1} \W{P} t^{-1}_{z_1}) \cdots) t^{-1}_{z_k})  \, \W{P}' \]
is obtained by a \emph{sequence} of partial conjugations by $t_{z_1}, \ldots, t_{z_k}$, which amounts to performing fibered Luttinger surgeries along Lagrangian tori or Klein bottles that are swept off by $z_i$ on the fibers over a loop on the base \cite{Auroux, BaykurLuttingerLF}. The reason we can break down the conjugation of $\W{P}$ by $\phi$ to a sequence of partial conjugations by Dehn twists is that at each step the conjugated mapping class \,$t_{z_{i}}\cdots t_{z_1} \W{P}\,t^{-1}_{z_1} \cdots t^{-1}_{z_{i}}$ indeed commutes with $t_{z_{i+1}}$. This is easy to see once we recall that as a mapping class $\W{P}= t_{C}^{-2} t_{C'}^2$ and that each $z_i$ is disjoint from $C$ and $C'$. The mapping class  $t_{z_{i}}\cdots t_{z_1} \W{P}\,t^{-1}_{z_1} \cdots t^{-1}_{z_{i}} =  t_{C}^{-2} t_{C'}^2$ certainly commutes with $t_{z_{i+1}}$.

Moreover, each $z_i$ is supported away from $\partial_1, \ldots, \partial_4$, so these surgeries can be viewed in the minimal models $X_{\phi}$. Since Luttinger surgery do not change $\omega$ away from these Lagrangians and since the canonical class $K$ can be supported away from them \cite{ADK}, $K^2$ or $K \cdot [\omega]$ do not change under these surgeries. Thus the symplectic Kodaira dimension does not change under these surgeries. This gives an alternative proof (without using Lemma~\ref{ruledno}) of all $(X_{\phi}, f_{\phi})$ being symplectic Calabi-Yau surfaces, since all are obtained from the symplectic Calabi-Yau $(X,f)$ we constructed in the previous section by partial conjugations. 

With this said, we observe that if our family $\{ X_{\phi} \}$ coincides with torus bundles over tori, we immediately get a proof of an improved version of a conjecture by Chung-I Ho and Tian-Jun Li: that every symplectic torus bundle over a torus can be obtained from the $4$-torus via Luttinger surgeries along tori \cite{HoLi}, or Klein bottles, as we add. (This conjecture is easier to verify for \emph{Lagrangian torus bundles} over tori, which however only constitute a subfamily of torus bundles with $b_1 \geq 3$.) All our examples are clearly obtained from $X$, our symplectic Calabi-Yau surface homeomorphic to the $4$-torus, via Luttinger surgeries. 
\end{remark}

%\begin{example}
%Maybe add the trefoil example. This should be homeomorphic to $S^1 \x M_K$ for $K$ the trefoil knot. 
%\end{example}

%\vspace{0.in}
\smallskip
% =========================================================
\subsection{The upper bound on the first Betti number of a genus-$g$ fibration} \ \label{Sec:bound}
% =========================================================

In \cite{KorkmazProblems}, Korkmaz asked if it is true that $b_1(X) \leq g$ for any non-trivial  genus-$g$ Lefschetz fibration $(X,f)$, noting that this is the case for all the examples he knows of ---with a single counter-example provided by Smith's genus-$3$ fibration on $T^4 \# 4 \CPb$. Here we will sketch a generalizing of our construction of a genus-$3$ Lefschetz pencil $(X,f)$ with $b_1=4$ in Section~\ref{Sec:4torus} to higher genera, which provides a negative answer to Korkmaz's question, for all odd $g> 1$. 

\begin{figure}[ht!]
 \centering
     \includegraphics[width=8cm]{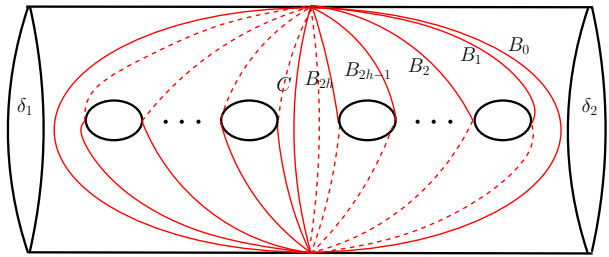}
     \caption{Vanishing cycles $B_j, C$ in Hamada's lift of the generalized Matsumoto fibration of genus $2h$.}
\vspace{0.2in}
     \label{GeneralizedMatsumoto}
\end{figure}

Our input will be the generalization of Matsumoto's genus-$2$ fibration constructed by Korkmaz himself \cite{Korkmaz} (and also by Cadavid \cite{Cadavid}): for every $h>0$, there is a genus-$2h$ Lefschetz fibration on the ruled surface $S^2 \x \Sigma_h \# 4 \CPb$. As shown by Hamada, the monodromy factorization of these can be lifted to $\Gamma_{2h}^2$ to obtain:
\begin{equation}
(t_{B_0} t_{B_1} t_{B_2} \cdots t_{B_{2h-1}} t_{B_{2h}} \, t_C )^2  t^{-1}_{\delta_1} t^{-1}_{\delta_2} = 1 \, ,
\end{equation}
where the curves $B_0, \ldots, B_{2h}, C$ and the boundary parallel curves $\delta_1, \delta_2$ are as shown in Figure~\ref{GeneralizedMatsumoto}. Embed this relation into $\Gamma_g$, for $g=2h+1$, by gluing along the two boundary components $\delta_1$ and $\delta_2$ of $\Sigma_{2h}^2$. We can also embed it by rotating the resulting surface by $180$ degrees as before. The product of the two factorizations we get in $\Gamma_{g}$ through these two embeddings prescribes a fiber sum decomposable \textit{achiral} genus-$g$ Lefschetz fibration. Removing pairs of matching Dehn twists with opposite signs (coming from the two embeddings of $t^2_C$ and $t_{\delta_1}= t_{\delta_2}$) gives us a genus-$g$ Lefschetz fibration $(\W{X_h},\W{f_h})$ with $H_1(\W{X_h}) \cong \Z^{2h+2}$, which can be seen from a straightforward calculation as before. Moreover, similar to our refined construction in Theorem~\ref{SCYfamily}, we could the further lifts of the Cadavid-Korkmaz fibration in $\Gamma_{2h}^4$ to obtain a pencil $(X_h, f_h)$ with $4$ base points, blowing-up which yields $(\W{X_h},\W{f_h})$. We have $b_1(X_h)=b_1(\W{X_h})= 2h+2=g+1 \nleq g$.  

With our marginal examples in mind, it seems reasonable to revamp Korkmaz's estimate and ask if every non-trivial Lefschetz pencil $(X,f)$ has $b_1(X) \leq g+1$. We can thus ask:

\begin{question} \label{b1RefinedQuestion}
Is the first Betti number $b_1(X)$ of any non-trivial genus-$g$ Lefschetz pencil $(X,f)$ bounded above by \,$g+1$? What is the supremum for $b_1(X)$ of all genus-$g$ Lefschetz pencils $(X,f)$ for a given $g$?
\end{question}

\noindent It has been known for a while by now that $b_1(X) \leq 2g-1$ for any non-trivial \linebreak genus-$g$ Lefschetz pencil/fibration $(X,f)$ \cite{ABKP}. The first part of Question~\ref{b1RefinedQuestion} asks if $g+1$ can be a sharper bound, still linear in $g$. The  second part of Question~\ref{b1RefinedQuestion}, which is formulated in the most general way, can already be answered for low genera pencils: for $g=1$ this supremum is clearly $0$, and it is $2$ for $g=2$, which can be easily shown using the Kodaira dimension (invoking \cite{SatoKodaira}, as in Remark~\ref{smallest}). For $g=3$, we conjecture that the supremum is indeed $4=g+1$. 

Note that, the analogous question for the supremum for the second Betti number $b_2(X)$ of \emph{relatively minimal} genus-$g$ \emph{pencils} $(X,f)$ is now fully answered. It is $11$ for $g=1$, realized by the pencil on $\CP \# \, 8 \CPb$; $37$ for $g=2$, following from Smith's analysis of genus-$2$ pencils in \cite{SmithGenus3}; whereas the supremum does \emph{not} exist when $g \geq 3$, which is implied by the arbitrarily long positive factorizations of the boundary multi-twist produced in \cite{BaykurVHM, DalyanKorkmazPamuk, BaykurMondenVHM}. However, constructions of fixed genus pencils with large $b_2$ is rather antagonistic to that of fixed genus pencils with large $b_1$: larger $b_2$ means more vanishing cycles (typically) killing $b_1$.

\newpage

% =========================================================

\end{document}